\UseRawInputEncoding
\documentclass[12pt]{iopart}
\usepackage{ntheorem}
\usepackage{enumerate}
 \usepackage{ragged2e}
 \usepackage{caption}
\usepackage{hyperref}

\usepackage{amssymb}
  \theoremsymbol{\mbox{$\blacksquare$}}
\usepackage{euscript,eufrak,verbatim}
\usepackage[dvips]{graphicx}
\usepackage[final]{epsfig}
\usepackage{url}
\newtheorem{example}{Example}
\newtheorem{theorem}{Theorem}
\newtheorem{corollary}{Corollary}

\newtheorem{lemma}{Lemma}
\newtheorem{proposition}{Proposition}

\newtheorem*{proof}{Proof}
\newtheorem{definition}{Definition}
\newtheorem{remark}{Remark}

\renewcommand{\theequation}{\thesection .\ \arabic{equation}}

\begin{document}

\title[Complete invariants and parametrization of expansive Lorenz maps]{Complete invariants and parametrization of expansive Lorenz maps}

\author{ Yiming Ding\& Yun Sun}

\address{  Center for Mathematical Sciences and Department of Mathematics, Wuhan University of Technology, Wuhan 430070, China}

\ead{dingym@whut.edu.cn and sunyun@whut.edu.cn }
\vspace{10pt}
\begin{indented}
\item[March 2021]
\end{indented}

\begin{abstract}
We obtain the complete conjugacy invariants of expansive Lorenz maps and for any given two expansive Lorenz maps, there are two unique sequences of $(\beta_{i},\alpha_{i})$ pairs. In this way, we can define the classification of expansive Lorenz maps. Moreover, we investigate the uniform linearization of expansive Lorenz maps through periodic renormalization.
\end{abstract}

%\maketitle
\section{Introduction }
It is well known that whether two maps $f:X\rightarrow X$ and $g:Y\rightarrow Y$ are  topologically conjugate or not depends on if there exists a non-decreasing homeomorphism $h:X\rightarrow Y$ such that $h\circ f=g\circ h$. If $h$ is monotonic and continuous but not necessarily invertible, then $f$ and $g$ are said to be semi-conjugate. Topological conjugation defines an equivalence relation which is useful in the theory of dynamical systems, since each class contains all maps sharing the same dynamics from the topological viewpoint.
\par  For maps on the circle, Poincare asked under which condition a given homeomorphism is equivalent (in some sense, e.g. topologically or smoothly) to some rigid rotation and proved that any orientation-preserving homeomorphism without periodic orbits is semi-conjugate to an irrational rotation. Denjoy proved that, when the rotation number is irrational, adding regularity to a given homemorphism $f$ ($\log Df$ is of bounded variation) is enough to guarantee topological conjugacy to a rotation. So for orientation-preserving homeomorphism without periodic orbits, the irrational rotation number is the complete topological invariant.  In 1958, Kolmogorov introduced an isomorphic invariant--entropy into ergodic theory, and gave a formula for the entropy of Bernoulli shifts which showed that all Bernoulli shifts are not isomorphic. In 1969, Ornstein[15] proved that two Bernoulli shifts are isomorphic if they have the same entropy. Entropy is the complete invariants for Bernoulli shifts. The main purpose of this paper is to present the complete topological invariants for expansive Lorenz maps via renormalization and uniform linearization.

\par A {\em Lorenz map} on $I=[0,1]$ is an interval map $f:I \to I$ such
that for some
$c\in (0,1)$ we have\\
\indent
(i)   $f$ is strictly increasing on $[0,c)$ and on $(c,1]$;\\
\indent (ii)  $\lim_{x \uparrow c}f(x)=1$, $\lim_{x \downarrow
	c}f(x)=0$.

If, in addition, $f$ satisfies the topological expansive condition

(iii) The pre-images set $C(f)=\cup_{n \ge 0}f^{-n}(c)$ of $c$ is dense
in $I$,

then $f$ is said to be an {\em expansive Lorenz map} [7-9].

Lorenz maps are one-dimensional maps with a single discontinuity, which
arise as Poincare return maps for flows on branched manifolds that model the strange
attractors of Lorenz systems. Lorenz map plays an important role in the study of the global
dynamics of families of vector fields near homoclinic bifurcations,
see [13],[14],[20],[21],[23] and
references therein.

Renormalization is a central concept in contemporary dynamics. The
idea is to study the small-scale structure of a class of dynamical
systems by means of a renormalization operator $R$ acting on the
systems in this class. A Lorenz map $f:I \to I$ is said to be renormalizable if there
is a proper subinterval $[a,\ b]$ and integers $\ell, r>1$ such that
the map $g: [a,\ b] \to [a,\ b]$ defined by
$$ \label{renormalization}
g(x)=\left \{
\begin{array}{ll}
f^{\ell}(x) & x \in [a,\  c), \\
f^{r}(x) & x \in (c,\  b],
\end{array}
\right.
$$
is itself a Lorenz map. The interval $[a, \ b]$ is called the renormalization interval. If $f$ is not renormalizable, it is said to be {\bf prime}.
\par It was proved in [4] that the minimal renormalization of renormalizable expansive Lorenz map always exists. We can define
a renormalization operator $R$ from the set of renormalizable expansive Lorenz maps to the set of expansive Lorenz maps [8][4].  For each renormalizable expanding Lorenz map, we define $Rf$ to be the minimal renormalization map of
$f$. For $n>1$, $R^{n}f=R(R^{n-1}f)$ if $R^{n-1}f$ is renormalizable. And
$f$ is $m$ ($0 \le m \le \infty$) {\it times renormalizable} if the renormalization process can proceed $m$ times
exactly.  For $0 <i \le m$, $R^{i}f$  is the $i$-th renormalization of $f$. \\

Linear mod one transformations
$T_{\beta,\alpha}(x) = \beta x + \alpha  (\ mod \ 1)$ play an important role in this work. $T_{\beta, \alpha}$ is said to be periodic if both $0$ and $1$ are periodic points of $T_{\beta, \alpha}$. We introduce some notations to identify the parameters.

\[ \Delta=\{(\beta, \alpha): \beta \in [1, 2], \alpha \in [0, 2-\beta]\}.\]
\[ \Delta_{pp}=\{(\beta, \alpha) \in \Delta, T_{\beta, \alpha}\ is \ prime \ and \ periodic\}.\]
\[ \Delta_{pe}=\{(\beta, \alpha) \in \Delta, T_{\beta, \alpha} \ is \ prime \ and \ expansive \}.\]
\[ \Delta_{r}=\{(\beta, \alpha)\in \Delta, T_{\beta, \alpha} \ is \ a \ rational \ rotation \}. \]%: \beta=1, \alpha=\frac{p}{q} \in [0, 1), (p,q)=1\}.

We emphasize that $\Delta$ contains all parameters $(\beta, \alpha)$ such that $T_{\beta, \alpha}$ is a Lorenz map. In what follows, when we talk about linear mod one transformation $T_{\beta,\alpha}$, it always means $(\beta, \alpha) \in \Delta $. If $\beta>1$, $T_{\beta, \alpha}$ is expansive. If $\beta=1$, $T_{\beta, \alpha}$ is a rotation $R_{\alpha}(x)=x+\alpha  (\ mod \ 1)$, when $\alpha$ is irrational, $R_{\alpha}$ is expansive, and it is non-expansive if $\alpha$ is rational.
So $T_{\beta, \alpha}$ is an expansive Lorenz map if and only if $(\beta, \alpha) \in \Delta \backslash \Delta_r$. $\Delta_{pp}$ is the set of parameters $(\beta, \alpha)$ so that $T_{\beta, \alpha}$ is prime and periodic. $\Delta_{pe}$ is the set of parameters $(\beta, \alpha)$ so that $T_{\beta, \alpha}$ is prime and expansive.
As we shall see in the main theorem, the piecewise linear Lorenz maps in $\Delta_{pp}$ and $\Delta_{pe}$ are building blocks of all  expansive Lorenz maps.

For $0 \le m \le \infty$, let $\mathcal{L}_m$ be the set of $m$-renormalizable expansive Lorenz maps, and
\[ \mathcal{A}_m=\underbrace{\Delta_{pp}\times \Delta_{pp} \times \cdots \times \Delta_{pp}}_m \times \Delta_{pe}, \ if \ m<\infty; \]
\[ \mathcal{A}_m=\underbrace{\Delta_{pp}\times \Delta_{pp} \times \cdots}_{\infty},   \ if \ m=\infty. \]
In other words, a sequence $A$ belongs to $\mathcal{A}_m$ is of the form $A=(A_1, A_2, \cdots, A_m, A_*)$ when $m<\infty$, and $A=(A_1, A_2, \ldots)$ when $m=\infty$, where $A_i=(\beta_i, \alpha_i) \in \Delta_{pp}$ for $i=1, \cdots, m$, $A_*=(\beta_*, \alpha_*) \in \Delta_{pe}$.

{\bf Main Theorem.} \emph {Fix $0\leq m\leq\infty$. There is a one-to-one correspondence between $\mathcal{L}_m$ and $\mathcal{A}_m$. More precisely, for $f \in \mathcal{L}_m$, there is a sequence $A\in \mathcal{A}_m$ associated with $f$, and for each sequence $B \in \mathcal{A}_m $, there exists an $m$-renormalizable expansive Lorenz map $g$ associated with $B$. Two expansive Lorenz maps are topologically conjugated if and only if they associated with the same sequence. Moreover, the cluster of points in the sequence are complete topological invariants for expansive Lorenz maps.}
%\begin{enumerate}[(1)]
%    \item  \emph{There are $m+1$ ordered pairs $(\beta_{i},\alpha_{i})(i=1,2\ldots m)$ and $(\beta_*, \alpha_*)$ associated with $f$. And two expansive Lorenz maps are topologically conjugate if and only if they admit the same ordered pairs.
%    \item  $(\beta_1,\alpha_1), \cdots, (\beta_m, \alpha_m), (\beta_*, \alpha_*)$ are completely invariants if $f$ is $m$ times renormalizable.}
%\end{enumerate}
\begin {remark}\rm
\
\begin{enumerate} [(1)]
     \item If $f$ is prime expansive Lorenz map with topological entropy $h(f)>0$, according to  a classical result of Parry, $f$ can be uniformly linearaized,  $f$ is conjugate to a linear mod one transformation
     $$ T_{\beta_*, \alpha_*}(x)=\beta_* x + \alpha_*  \ \  \ mod \ \  1 \ \ for \ x \in [0,1],\ 2>\beta>1, \ \alpha\in [0,\ 2-\beta_*],$$
     where $\beta_*=\exp(h(f))$, and $\alpha_*$, as we shall see in Section 4, can be regarded as the generalized rotation number.   So there is a one-to-one corresponding between prime $f$ and $(\beta_*,\alpha_*)$.
     \item If $f$ is $m-$renormalizable, $(\beta_1,\alpha_1), \cdots, (\beta_m, \alpha_m)$ come from the consecutive renormalizations of $f$.
     \item We emphasize that the ordered  pairs $(\beta_1,\alpha_1), \cdots, (\beta_m, \alpha_m), (\beta_*, \alpha_*)$ are complete topological invariants of expansive $m-$renormalizable Lorenz maps, and the order of the pairs can not be changed. In fact, if the ordered pairs associated with two 2-renormalizable expansive Lorenz maps $f$ and $g$  are  $(1, 0.3), (1, 0.4),(1.3, 0.5) $ and $(1, 0.4), (1, 0.3),(1.3, 0.5)$, respectively, then $f$ and $g$ are not topologically conjugated.
     \item The ordered pairs can not only be used to check whether two expansive Lorenz maps are topologically conjugate or not, but also be used to describe the distance between two expansive Lorenz maps which are not conjugate. See Section 4 for details.

         %Denote $d(f,g)$ as the distance of two expansive Lorenz maps $f$ and $g$, let $p$ be the maximal integer such that the first $p$ pairs of $f$ and $g$ are identical, then we can obtain:
%          $$
%             d(f,g)=\frac{1}{2^{p}}(1+\frac{1}{2^{s_{+}}}+\frac{1}{2^{s_{-}}})
%               $$
%             where $s_{+}$ means the first $s_{+}$ words of $k_{+}^{p+1}$ and $k_{+}^{p+1^{'}}$ are identical, $s_{-}$ means the first $s_{-}$ words of $k_{-}^{p+1}$ and $k_{-}^{p+1^{'}}$ are the same.
     \item It is an structural theorem for expansive Lorenz maps. Let $f$ be an $m$-renormalizable expansive Lorenz map $ (0\leq m\leq\infty)$, then $f$ is uniformly linearizable if and only if $m$ is finite and $\beta_{i}=1$ for $i=1,2\ldots m$. Moreover, $h_{top}f=0$ if and only if $\beta_{i}=1$ for all $i$.
\end{enumerate}
\end{remark}
\par The remain parts of the paper are organized as follows: In Section 2, we consider the renormaliation of expansive Lorenz maps, using topological and combinatorial approaches, and prove the Factorization theorem. In Section 3, we show how a prime Lorenz map corresponds to a point in $\Delta \backslash \Delta_r$, and provide the formulae to calculate $\beta$ and $\alpha$. In Section 4, we show that an expansive Lorenz map can be associated with a cluster of points in $\Delta$, which is the completed topological invariants of expansive Lorenz maps. A combinatorial metric is defined on the space of expansive Lorenz maps by the cluster of points.  In the last Section, we characterize some special classes of expansive Lorenz maps via the cluster of points.
\section{Renormalization }

The idea of renormalization for Lorenz map was
introduced in studying simplified models of Lorenz attractor,
apparently firstly in Plamer [16] and Parry [18] (cf.
[5]). The renormalization operator in Lorenz map family, is
the first return map of the original map to a smaller interval
around the discontinuity, rescaled to the original size. Glendinning
and Sparrow [8] presented a comprehensive study of the
renormalization by investigating the kneading invariants of
expanding Lorenz map. Ding proposed a topological approach by using completely invariant closed set to investigate the renormalization of expansive Lorenz maps [4].

\subsection{Topological approach}
We briefly introduce the topological approach for renormalization of expansive Lorenz map developed in [4]. Recall that a subset $E$ of $I$ is completely invariant under $f$ if
$f(E) = f^{-1}(E) = E$, and it is proper if $E \neq I$.

\begin{lemma}{\bf([4, Theorem A])}.\label{renormalization} {\it Let $f$ be an expansive Lorenz map. There is
a one-to-one correspondence between the renormalizations and proper
completely invariant closed sets of $f$. More precisely,  suppose $E$ is a proper completely invariant closed set of $f$, put
$$e_- = \sup{x \in E, x < c}, \ \ \ \ e_+ = \inf{x \in E, x > c}$$
and
$$\ell= N((e_- , c)),\ \ \ \ \  r = N((c, e_+ )).$$
Then
$f^{\ell}(e_-)=e_-,\  f^r(e_+)=e_+$,
and the following map $R_Ef$ is a renormalization  of $f$
\begin{equation} %\label{completely invariant renormalization}
R_Ef(x)=\left \{ \begin{array}{ll}
f^{\ell}(x) & x \in [f^{r}(c_+),\   c) \\
f^r(x) & x \in (c, \ f^{\ell}(c_-) ]
\end{array}
\right.
\end{equation}
On the other hand, if $g$ is a renormalization of $f$, then there exists a unique
proper completely invariant closed set $B$ such that $R_B f = g$.}
\end{lemma}

According to the Lemma 1, a possible way to characterize the renormalizability is to look
for the minimal completely invariant closed set $D$ of $f$, in the sense that $D\subseteq E$
for each completely invariant closed set $E$ of $f$.  If we can find the minimal completely invariant closed set, then
$f$ is renormalizable if and only if $D \neq I$.

\begin{lemma}{\bf([4, Theorem B])}.\label{renormalization} Let $f$ be an expansive Lorenz map with
minimal period $\kappa(1<\kappa<\infty)$.  Then we have the following:
\begin{enumerate}

\item $f$ admits a unique $\kappa$-periodic orbit $O$.
\item $D=\overline{\bigcup_{n\ge 0}f^{-n}(O)}$ is the minimal completely invariant closed set of $f$.
\item $f$ is renormalizable if and only if $D \neq I$. If $f$ is
renormalizable, then $R_D$, the renormalization associated to $D$,
is the minimal renormalization of $f$.
\item The following {\bf trichotomy holds: i) $D=I$, ii)$ D=O $, iii) $D$ is a Cantor set.}
\end{enumerate}
\end{lemma}
It is easy to see the cases $\kappa=1$  and $\kappa=\infty$ are prime, Theorem B describes the
renormalizability of expanding Lorenz map completely.

According to Lemma 2, a renormalizable expanding Lorenz map admits a
unique minimal renormalization. We can define a renormalization operator $R$ from
the set of renormalizable expansive Lorenz maps to the set of expandsive Lorenz maps.
 For each renormalizable expanding Lorenz map, we define $Rf$
to be the minimal renormalization map of $f$. For $n>1$, $R^n f = R(R^{n-1} f)$ if
$R^{n-1}f$ is renormalizable. And $f$ is $m (0\le m \le \infty)$ times renormalizable if the
renormalization process can proceed $m$ times exactly. For $0< i \le m$, $R^i f$ is the
$ith$ renormalization of $f$.

The minimal renormalization is said to be {\it periodic} if the minimal completely invariant closed set $D=O$, where
$O$ is the periodic orbit with minimal period of $f$. And the $ith$ renormalization $R^i f$
is periodic if it is a periodic renormalization of $R^{i-1} f$.  The periodic renormalization is interesting because the linear mod one transformation
$T_{\beta,\alpha}(x) = \beta x + \alpha  (\ mod \ 1), \ \  \beta\in [1, \ 2],\  \alpha \in [0,\ 2-\beta].$
can only be periodically renormalized (see [6], [3]) . This
kind of renormalization was studied by Alsed$\grave{a}$ and $Falc\grave{o}$, Malkin. It was
called phase locking renormalization by Alsed$\grave{a}$ and $Falc\grave{o}$  because it appears
naturally in Lorenz map whose rotational interval degenerates to a rational point.

\subsection{Combinatorial approach}

Now we try to translate the results in previous subsection to combinatorial approach, which is a little bit different from the approach of Glendinning and Sparrow.

Let $f$ be an expansive Lorenz map. The trajectories of points in $I$ by $f$ can be coded by elements of $\Omega=\{0,\ 1\}^{\infty}$,
which denotes the set of infinite words( or sequences) ${\bold a}=a_1a_2\cdots a_n\cdots$ on the alphabet $\{0,\ 1\}$.
The lexicographic order on $\Omega$ is the total order defined $a<b$ if $a\neq b$ and $a_n<b_n$, where $n$ is the least index such that $a_n\neq b_n$.

 The kneading
sequence of a point $x$, $k_f(x)$, is defined to be $\epsilon_0
\epsilon_1 \ldots$ of $0's$ and $1's$ as follows:
$$ \epsilon_i=0\ \ \ \ \ \ if\ \ \ f^i(x)<c \ \ \ \ and \ \ \
\epsilon_i=1\ \ \ \ \ \ if\ \ \ f^i(x)>c.$$
This definition works for $x \notin C(f)=\cup_{n \ge 0}f^{-n}(c)$. In the case where $x$ is
a preimage of $c$, $x$ has upper and lower kneading sequences
$$ k_{f,+}(x)=\lim_{y\downarrow x}k_f(y), \ \ \ \ \ \ \ \ \
k_{f,-}(x)=\lim_{y\uparrow x}k_f(y),$$ where the $y's$ run through
points of $I$ which are not the preimages of $c$. In particular,
$k_+=k(c+)=1k(0)$ and $k_-=k(c-)=0k(1)$ are called the kneading
invariants of $f$, which were used to developing the renormalization
theory of expansive Lorenz map by Glendinning and Sparrow [8].
The kneading space of $f$ is
$$K(f)=\{k_f(x): x \in I\}.$$
Let $\sigma$ be the shift map, operating on the kneading space
$K(f)$, then clearly $k_{f}(f(x))=\sigma(k_f(x)))$, with similar
results holding for the upper and lower kneading sequences of points
$x$ which are pre-images of $c$. The dynamics of $f$ on
$I$ can be modeled by the shift map $\sigma$ on kneading space.

\begin{lemma}{\bf([12, Theorem 1])}.\label{kneading space}
If $f$ is a topologically expansive Lorenz map, then the kneading invariant $K(f)=(k_{+},k_{-})$, satisfies
	\begin{equation} \label{expanding}
		\sigma (k_{+})\le \sigma^n(k_{+})<\sigma (k_{-}), \ \ \ \ \ \ \sigma (k_{+})< \sigma^n(k_{-})\le\sigma (k_{-}) \ \ \ \  n \ge 0,
	\end{equation}
 Conversely, given any two sequences $k_{+}$ and $k_{-}$ satisfying (2.2), there exists an expansive Lorenz map $f$ with $(k_{+},k_{-})$ as its kneading invariant, and $f$ is unique up to conjugacy by a homeomorphism of $I$.
\end{lemma}

\par In  Lemma 3, we know the one to one
correspondence between the pair $(k_{+},k_{-})$ and the
kneading space of Lorenz map. Two expansive Lorenz maps are said to be
 equivalent, if they admit the same kneading invariant. If
$K(f)=\Omega_{k_{+},k_{-}}$, we call $\Omega_{k_{+},k_{-}}$ the
kneading space of $f$. When we do not want to mention $f$, we call
$\Omega_{k_{+},k_{-}}$ a kneading space. Since the coding of the trajectories of expansive $f$ is a homepmorphism which preserves the lexicographic order, we know that $K(f)=\Omega_{k_{+},k_{-}}=\{{\bold a} \in \Omega: \sigma(k_+) \leq \sigma^n {\bold a} \leq \sigma(k_-)\}$.

\subsection{Renormalization of expansive Lorenz maps}

\noindent  Now we consider the renormalization of Lorenz map in combinatorial
way. The following definition is essentially from Glendinning and Sparrow [8].

\begin{definition} \label{com-renormal}
\rm We say that the kneading invariant $K=(k_+, k_-)$ is renormalizable if there exists finite, non-empty words $(w_+,\ w_-)$, such that
\begin{equation} \label{com-renormal1}
\left \{ \begin{array}{ll}
k_{+}=1k(0) =&w_+ w_-^{p_1} w_+^{p_2} \cdots,\\
k_{-}=0k(1) =&w_- w_+^{m_1} w_-^{m_2} \cdots,
\end{array}
\right.
\end{equation}
where $w_+=1\cdots$, $w_-=0\cdots$ and the lengths $|w_+|>1$
and $|w_-|> 1$, and $n_1, m_1>0$. The kneading invariant of the renormalization is $RK=(k_+^1, k_-^1)$, where
\begin{equation} \label{com-renormal2}
\left \{ \begin{array}{ll}
k_{+}^{1}=1k(0)_{1}=&1 0^{p_1} 1^{p_2} \cdots,\\
k_{-}^{1}=0k(1)_{1}= & 0 1^{m_1} 0^{m_2} \cdots.
\end{array}
\right.
\end{equation}
\end{definition}

\vspace{1cm}
It is obvious that Definition [1] is equivalent to the Definition 1, Lorenz map $f$ is renormalizable is equivalent to its kneading invariant
$K(f)=(k_{+},k_{-})$ is renormalizable. If $K=(k_+, k_-)$ is not renormalizable it is said to be prime.  Note that we do not involve $(w_+, w_-)=(1, 01)$ and $(w_+, w_-)=(10, 0)$ in the definition of renormalization as in [8]. The cases $(w_+, w_-)=(1, 01)$ and $(w_+, w_-)=(10, 0)$  correspond to trivial renormalizations.

If $K$ is renormalizable, then by the substitution $(w_+, w_-)\mapsto (1, 0)$ we obtain a new kneading invariant denoted by $RK$. To describe the renormalization more concisely, we use $*$-product, which is introduced for kneading invariant of Lorenz map in 1990s [19].
Let $W=(w_+^{\infty}, w_-^{\infty})$ be an admissible periodic kneading invariant, and $K^1=(k_+^1, k_-^1)$ be an admissible kneading invariant.
The $*$-product of $W$ and $K$ is defined to be a new kneading invariant pair $K=(k_+, k_-):=W*RK^1$
\begin{equation}\label{product}
(k_{+},k_{-})=(w_+,\ w_-)*(k_{+}^1,k_{-}^1)
\end{equation}
where $(k_{+},k_{-})$ is the pair of sequence obtained by
replacing $1$'s by $w_+$, replacing $0$'s by $w_-$ in $k_{+}^1$ and $k_{-}^1$.

Using $*$-product, (2.3) and (2.4) can be expressed by $K=W*RK$. So a kneading invariant is renormalizable if and only if it can be decomposed as the $*$-product of an admissible periodic kneading invariant and another admissible kneading invariant. If $K$ is prime, then it can not be decomposed as the $*$-product of an admissible periodic kneading invariant and another admissible kneading invariant. One can check that the $*$-product satisfies the associative law, but does not satisfy the commutative law in general.

It is known that if both $W=(w_+, w_-)$ and $K=(k_+, k_-)$ are admissible, then the $*$-product $W*K$ is admissible. Suppose $W_i=(w_{i,+}, w_{i,-}), i=1, \ldots, m$ are admissible periodic kneading invariant, one can define the $*$-product: $W_1*W_2*\cdots *W_m*K$, which is admissible if $K$ is also admissible. So we can obtain different kinds of kneading invariants by $*$-product. An interesting question is whether a kneading invariant can be decomposed as the  $*$-product of prime kneading invariant.

\begin{remark}\rm

\
\begin{enumerate}[(1)]
\justifying
%\item If $\Sigma_{\alpha,\ \beta}$  is renormalizable with finite
%words $(w_+,\ w_-)$ , then the sequence $e_+=w_+^{\infty}$ and
%$e_-=w_-^{\infty}$ are periodic points in $\Sigma_{\alpha,\ \beta}$.
%Where $e_+$ and $e_-$ are the kneading sequence of the same points
%in Theorem A. The renormalization is periodic if $w_+^{\infty}$ and
%$w_-^{\infty}$ belong to the same periodic orbit.
\item $(k_{+}^{1}, k_{-}^{1})$ is obtained by replacing $w_+$ and $w_-$ by $0$ and $1$ in $(k_{+},k_{-})$. So if $\Sigma_{k_{+},k_{-}}$  is renormalizable with finite words $(w_+,w_-)$, one can
define a map $R_{w_+, w_-}(k_{+},k_{-})=(k_{+}^{1}, k_{-}^{1})$.
\item If $\Sigma_{k_{+},k_{-}}$  is prime, then
$(\Sigma_{k_{+},k_{-}},\sigma)$  is {\it l.e.o.} in the sense
that $\bigcup_{i =0}^n \sigma^i(U)=\sum_{k_{+},k_{-} }$ for every
cylinder set $U \subset \Sigma_{k_{+},k_{-}}$.

\end{enumerate}

\end{remark}

A subset $E$ in  $\Omega_{k_{+},k_{-}}$ is said to be completely invariant with respect to
$(\Omega_{k_{+},k_{-}},\ \sigma)$, if
\begin{equation} \label{total invariant}
 \sigma(E)=\sigma^{-1}(E)\cap \Omega_{k_{+},k_{-}}=E.
\end{equation}
We also call $E$ a completely invariant subset of $\Omega_{k_{+},k_{-}}$ and it is proper if $\emptyset \neq E\neq \Omega_{k_{+},k_{-}}$.
In other words, $E$ is a  completely invariant subset of $\Omega_{k_{+},k_{-}}$ if and only if $\forall {\bold a} \in E, {\bold b} \in \Omega_{k_{+},k_{-}}$ such that $\sigma {\bold b}={\bold a}$, we have $\sigma {\bold a} \in E$ and ${\bold b} \in E$.

The following Lemmas 1 follow from Lemma 1 and Lemma 2.  We provide a combinatorial proof of Lemma 3 in Appendix.

\begin{lemma}

Suppose $\Omega(f)=\Omega_{k_{+},k_{-}}$, $E$ is a proper completely
invariant closed set of $\Sigma_{k_{+},k_{-}}$. Put
\begin{equation}\label{critical}
e_-:=\sup \{\delta \in E, \delta < k_{-}\} \ \ \ \ \ e_+:= \inf
\{\delta \in E,\delta > k_{+}\}
\end{equation}
$l$ and $r$ be the maximal integers so that $f^{l}$ and $f^{r}$ is continuous on $(e_{-},k_{-})$ and $(e_{+},k_{+})$, respectively.
and
$$
w_-:=(k_{-})_{\ell}, \ \ \ \ \ \ w_+:=(k_{+})_{r}
$$
which means the first $l$ words of $k_{-}$ and the first $r$ words of $k_{+}$, then we have
\begin{enumerate}[(1)]
\item  $e_-=w_-^{\infty}$ and $e_+=w_+^{\infty}$, i.e., $e_-$ and
$e_+$ are periodic;

\item  The following decompositions holds:
\begin{equation} \label{comp}\Omega_{k_+,\
k_-}\backslash E= \bigcup_{n \ge 0}\sigma^{-n}((e_-,\ e_+))=
\bigcup_{n \ge 0}\sigma^{-n}((\sigma^r(k_{+}),\
\sigma^{\ell}(k_{-})).
\end{equation}

\item The following equality holds:
$$ E=\Omega_{w_{+}^{\infty},w_{-}^{\infty}}.$$

\item $K=({k_{+},k_{-}})$ can be renormalized by words $w_+$ and
$w_-$.

\end{enumerate}
\end{lemma}

\vspace{1cm}

\begin{corollary}\label{renomal-totaly invariant}
  $K=({k_{+},k_{-}})$ can be renormalized by $(w_+,\
w_-)$ if and only if $\Sigma_{w_+^{\infty},
w_-^{\infty}}$ is a proper completely invariant closed set of
$\Omega_{k_{+},\ k_{-}}$.
\end{corollary}

\begin{proof}
\rm According to Theorem 1, $\Omega_{w_+^{\infty},w_-^{\infty}}$ is a proper completely invariant closed set of
$\Omega_{k_{+},k_{-}}$ implies   $\Omega_{k_{+},\ k_{-}}$ can be
renormalized by $w_+$ and $ w_-$.

Now we suppose $\Omega_{k_{+},k_{-}}$ can be renormalized by
$w_+$ and $ w_-$, with renormalization interval $[a, b]$. Put
$$F_{a, b}=\{\delta \in \Omega_{k_{+}, k_{-}}, orb(\delta)\cap (a,b)\neq \emptyset \},$$
$$E_{a,b}=\{\delta \in \Omega_{k_{+}, k_{-}}, orb(\delta)\cap (a,b)= \emptyset \}.$$
Since $F_{a, b}=\bigcup_{n \ge 0}\sigma^{-n}((a, b))$, $F_{a, b}$
is a completely invariant open set. And $E_{a, b}=\Sigma_{k_{+}, k_{-}} \backslash F_{a, b}$ is a completely invariant closed set of
$\Omega_{k_{+}, k_{-}}$. The renormalization associated with
$E_{a, b}$ is just the renormalization with words $(w_+, w_-)$
follows from  Theorem 1. $\hfill\square$
\end{proof}

\begin{remark}\label{two-lorenz}\rm

\
\begin{enumerate} [(1)]
\item There is a one-to-one correspondence between the
renormalizations and proper complete invariant closed set of
$\Omega_{k_{+}, k_{-}}$.
\item {Reormalization is essentially a reduction mechanism. If
$(\Omega_{k_{+},k_{-}},\sigma)$ is renormalizable with finite
words $w_+$ and $w_-$, then its dynamics is captured by two Lorenz
maps: the renormalization map $(\Omega_{k_{+}^{1},k_{-}^{1}},
\sigma)$ and the dynamics on the corresponding completely invariant closed set:
$(\Omega_{ w_+^{\infty} ,  w_-^{\infty} },\ \sigma)$.}
\end{enumerate}
\end{remark}

A possible way to describe the renormalizability of $\Omega(f)=\Omega_{k_{+},k_{-}}$ is to look
for the {\it minimal completely invariant closed set} of $\Omega_{k_{+},k_{-}}$. The
minimal completely invariant closed set relates to the periodic
orbit with minimal period of $(\Omega_{k_{+},k_{-}},\sigma)$. Suppose the minimal period of the
periodic points of it is $\kappa$. It is easy to see that $K(f)$ is
prime if $\kappa=1$ or $\kappa=\infty$. If $1<\kappa<\infty$, then
$(\Omega_{k_{+},k_{-}},\sigma)$ admits unique $\kappa$-periodic orbit $O=\{w^{\infty}, \sigma(w^{\infty}), \cdots, \sigma^{\kappa-1}(w^{\infty})\}$. Put
$D=\overline{\bigcup_{n\ge 0}\sigma^{-n}(O)}$. Then we have the following
Lemma 5.
%\begin{definition}
%A renormalization is said to be periodic if $w_-^{\infty}$ and
%$w_+^{\infty}$ belong to the same periodic orbit.
%\end{definition}

\justifying
\begin{lemma}\label{pr}
 Let $f$ be an expansive Lorenz map with minimal period $\kappa$, $1<\kappa<\infty$, $K(f)=(k_+, k_-)$ then we have:
\begin{enumerate}[(1)]
\item $D$ is the minimal completely invariant closed set of $(\Omega_{k_{+},k_{-}},\sigma)$.
\item $f$ is renormalizable if and only if $D \neq \Omega_{k_{+},k_{-}}$. If $f$ is
renormalizable, then $R_D$, the renormalization associated to $D$,
is the unique minimal renormalization of $K(f)$.
\item The following trichotomy holds {\it trichotomy:}
 i) $D=\Omega_{k_{+},k_{-}}$, ii) $D=O$, iii) $D$ is a Cantor set.
\end{enumerate}
\end{lemma}
\begin{lemma}Let $W=(w_{+}^{\infty},w_{-}^{\infty})$ be an admissible periodic kneading invariant, then $W*K^{*}$ is expansive if and only if $K^{*}$ is expansive.
\begin{proof}\rm By the associative law of star product, we can write $K=W*K^{*}=(k_{+},k_{-})$, where $K^{*}=(k_{+}^{*},k_{-}^{*})$ is admissible expansive kneading invariant. It follows that
\begin{equation}
\left \{ \begin{array}{ll}
k_{+}=&w_{+} w_{-}^{p_1} w_{+}^{p_2} \cdots,\\
k_{-}=&w_{-} w_{-+}^{m_1} w_{-}^{m_2} \cdots,
\end{array}
\right.
\end{equation}
Suppose there is an integer $i>1$ such that $\sigma^{i}(k_{+})=\sigma(k_{-})=\sigma(w_{-} w_{-+}^{m_1} w_{-}^{m_2} \cdots)$. Since $(k_{+},k_{-})$ are admissible kneading invariant, they satisfy (2.2), $\sigma(k_{-})$ is the largest element in the kneading space $\Omega_{k_{+},k_{-}}$, which indicates that the preimage of $\sigma^{i}(k_{+})$ is $\sigma^{i-1}(k_{+})=0\sigma(k_{-})=w_{-} w_{-+}^{m_1} w_{-}^{m_2} \cdots$. It follows that there exists positive integer $j$ such that $\sigma^{j}(k_{+}^{*})=\sigma(k_{-}^{*})$, which contradicts to the expansiveness of $(k_{+}^{*},k_{-}^{*})$. Hence, such an $i$ can not exist. By similar arguments, there is no positive integer $i'$ such that $\sigma^{i'}(k_{-})=\sigma(k_{+})$. So $(k_{+},k_{-})$ satisfy the Hubbard-Sparrow condition (2.2). Using Lemma 3, there exists expansive Lorenz map $g$ with kneading invariant $(k_{+},k_{-})$, and the expansive Lorenz map is unique up to topological conjuacy. $\hfill\square$
\end{proof}
\end{lemma}

\begin{definition}\label{consecutive renormalization}
\rm Let $f$ be an expansive Lorenz map. The minimal renormalization is said to be periodic if the minimal completely invariant closed set $D=O$, where $O$ is the periodic orbit with minimal period of $f$. And the $i$th renormalization $R^{i}f$ is periodic if it is a periodic renormalization of $R^{i-1}f$.
\end{definition}
\par The periodic renormalization is interesting because
$\beta$-transformation can only be renormalized periodically (see [3][6])
. This kind of renormalization was studied by
Alsed$\grave{a}$ and Falc$\grave{o}$ [1], Malkin [11]. It
was called phase locking renormalization in [1] because it
appears naturally in Lorenz map whose rotational interval
degenerates to a rational point.

\begin{theorem}{\bf(Factorization Theorem)}.\label{renormalization}\label{decomposition of kneading invariant}
\rm An admissible kneading invariant $K=(k_+, k_-)$ can be factorized uniquely to the $*$-products of  prime kneading invariants.
 More precisely,

\begin{enumerate}[(1)]
\item If $K$ is finitely renormalizable, there is positive integer $m$ such that $R^mK:=K^*=(k^*_+, k^*_-)$ is prime. There are $m$  pair or finite words $W_i=(w_{i,+}, w_{i,-})$ so that $W_i^{\infty}=(w_{i,+}^{\infty}, w_{i,-}^{\infty})$ ($i=1, \cdots, m$)is  periodic prime kneading invariant such that
 $$K=W_1*\cdots*W_m*K^*.$$
 \item If $K$ is infinitely renormalizable, then there is a sequence of $W_i=(w_{i,+}, w_{i,-})$ so that $W_i^{\infty}=(w_{i,+}^{\infty}, w_{i,-}^{\infty})$ ($i=1, \cdots, $) is admissible periodic prime kneading invariant such that
$$K=W_1*W_2*W_3*\cdots.$$
\end{enumerate}
\end{theorem}

\begin{proof}
\rm Let $K=(k_+, k_-)$ is an admissible kneading invariant. If $K$ is prime, the proof is completed. If $K$ is not prime, then by Lemma *, the Lorenz shift $(\Omega_K, \sigma)$ admits a unique periodic orbit $O_1$ with minimal period, and a minimal completely invariant closed set $D_1$.
Put $d_-=\sup\{{\bold a} \in D_1, {\bold a}<k_-\}$,  $d_+=\sup\{{\bold a} \in D_1, {\bold a}>k_+\}$. We know that both $d_-$ and $d_+$ are periodic points of $(\Omega_K, \sigma)$. There are finite words $w_{1,-}$ and $w_{1,+}$ such that $d_-=w_{1,-}^{\infty}$ and $d_+=w_{1,+}^{\infty}$. And $K$ can be renormalized by $W_1=(w_{1,+}, w_{1,-})$. According to the definition of $*$-product, we have $K=W_1*RK$.

Now we consider the dynamics on the minimal completely invariant closed set $D_1$. By Lemma *, if $K$ is renormalizable,  we have either $D_1=O_1$ or $D_1$ is a Cantor set in $\Omega_K$.

If $D_1$ is a Cantor set in  $\Omega_K$, since $d_-$ is the maximum point among all points in $D_1$ which are less than $k_-$, and $d_+$ is the minimum point among all points in $D_1$ which are greater than $k_+$, we conclude that $\sigma d_-$ is the maximum point and  $\sigma d_+$  is the minimum point in $D_1$. It follows that $W_1^{\infty}=(w_{1,+}^{\infty}, w_{1,-}^{\infty})$ is admissible periodic kneading invariant. The minimality of $D_1$ indicates that $(\Omega_K, \sigma)$ is locally eventually onto, it is prime. So $W_1^{\infty}$ is an admissible periodic prime kneading invariant.
%the restriction of $(\Omega_{K}, \sigma)$ on $D_1$ is $(\Omega_{D_1}, \sigma)$ is can be characterized by the kneading invariant $W_1^{\infty}=(w_{1,+}^{\infty}, w_{1,-}^{\infty})$, which is an admissible prime kneading invariant because $D_1$ is the minimal completely invariant closed  set in $(\Omega_{K}, \sigma)$.
Observe that both $w_{1,+}^{\infty}$  and $ w_{1,-}^{\infty}$ are periodic points,  $(\Omega_{D_1}, \sigma)$ is a sub-shift of finite type.
If $D_1=O_1$, i.e., the minimal renormalization is periodic, then both $d_-$ and $d_+$ belong to $O_1$, so $w_{1,+}^{\infty}$ and $w_{1,-}^{\infty}$ belong to the same periodic orbit $O_1$ with rotation number $p/q$, and $(p, q)=1$. By results in [6] or [3], the periodic orbit $O_1$ is a well-ordered periodic orbit with rotation number $p/q$, and $(\Omega_{O_1}, \sigma)$ is prime. In this case, the kneading invariant $W_1^{\infty}=(w_{1,+}^{\infty}, w_{1,-}^{\infty})$ is not admissible in the sense of Hubbard and Sparrow, because there is no expansive Lorenz map can realize the kneading invariant. There is still Lorenz map, $g(x)=x+p/q, \ mod \ \ 1$, whose kneading invariant is $W_1^{\infty}=(w_{1,+}^{\infty}, w_{1,-}^{\infty})$. Note that $g(x)$ is a rational rotation, which is not topologically expansive.
If $RK$ is prime, put $Rk=K^*$, we obtain $K=W_1*K^*$, the factorization is completed. If $RK$ is not prime, repeat the above arguments, one can obtain the minimal completely invariant closed set $D_2$ of $(\Omega_{RK}, \sigma)$, and obtain a pair of words $W_2=(w_{2,+}, w_{2,-})$ so that $RK=W_2*R(RK)=W_2*R^2K$, and $W_2^{\infty}=(w_{2,+}^{\infty}, w_{2,-}^{\infty})$ is an admissible periodic prime kneading invariant. Combine the first two renormalizations, we obtain $K=W_1*W_2*R^3K$. By the same arguments, we know that the restriction of $(\Omega_{RK}, \sigma)$ on $(\Omega_{D_2}, \sigma)$ is a sub-shift of finite type if the renormalization is not periodic, and it is a well-ordered periodic orbit with rational rotation number if the renormalization is periodic.

Repeat the procedure if necessary. If $K$ can be renormalized finite times, there exists positive integer $m$ such that $R^mK$ is prime. Put $R^mK=K^*$, we conclude that $K=W_1*\cdots*W_m*K^*$, and $W_i^{\infty}$  $( i=1,\cdots, m)$ are admissible periodic prime kneading invariants.
If $K$ can be renormalized infinite times, we can obtain a sequence of admissible periodic prime kneading invariants $\{W_i^{\infty}\}$ ($i=1, 2,\cdots $) such that $K=W_1*W_2*W_3*\cdots$.
$\hfill\square$

\end{proof}

\section{Linearization of prime expansive Lorenz maps}

In the one-dimensional dynamics, Ulam asked the following piecewise linearizing question: ``Is every `smooth' function $f(x)$ on $(0,1)\rightarrow(0,1)$ conjugate to a suitable piecewise linear function?" He emphasized that ``An affirmative answer to the above question would reduce the study of the iteration of such functions to $f(x)$ to a purely combinatorial investigation of the properties of `broken line' functions." [22]. Limited work has been done on the conjugacy to a linear mod one transformation [6]. It would be interesting to ask the following uniform piecewise linearizing question: Given a map $f$ on $[0,1]\rightarrow[0,1]$, is it possible to find a non-decreasing homeomorphism $h$ and a linear mod one transformation $T$ such that $h\circ f=g\circ h$? In this paper, we are concerned with the linearization of expansive Lorenz maps in combinatorial way.
\par As one of the simplest uniform piecewise linear map on the interval, the linear mod one transformation [5] defined by
$$T_{\beta, \alpha}(x)=\beta x+\alpha \ \ \ mod \ \ 1$$
has attracted considerable attention. When $1 < \beta
\leq 2$, $\alpha+\beta\leq 2$,  $T_{\beta, \alpha}(x) $ has one discontinuity at $x=(1-\alpha)/\beta$.
 $\beta$-transformation form an explicit two-parameter family of a dynamical systems whose complicated behavior can be completely determined by their coordinates $(\beta,\alpha)$ in parameter space. A simple extension of an argument of Milnor and Thurston [12] shows that every Lorenz map $f$ is semi-conjugate to a $\beta$-transformation; furthermore semi-conjugacy preserves topological entropy. Hubbard and Sparrow [6] implies that the semi-conjugacy can be full topological conjugacy if and only if $f$ has the same kneading invariant as some $\beta$-transformation $T(x)$. According to Parry [17], $f$
is conjugate to a $\beta$-transformation if $f$ is strongly transitive, and an expanding Lorenz map is strongly transitive if and only if it is prime. Hence we can know an expansive Lorenz map corresponds to a pair $(\beta,\alpha)$ if it is Prime, where $(\beta,\alpha)$ means the parameters of corresponding $\beta$-transformation.

\subsection{Calculation of $\alpha$ and $\beta$}
In 1988, Milnor and Thurston [12] have studied that if $f:[0,1]\rightarrow [0,1]$ is a unimodal map, then its topological entropy is related to the smallest positive zero $s$ of a certain power series by $h(f)=\log(1/s)$. In 1996, Paul Glendinning and Toby Hall [7] proved a similar result for Lorenz maps. In 2014, work directly with the symbolic space and do not require it to be the address space of some map, Barnsley [2] put forward the following result.
\begin{lemma}{\bf([2],Lemma 3)} Let $k_{-}=a_{1}a_{2}\cdots$, $k_{+}=b_{1}b_{2}\cdots\in\{0,1\}^{N}$, with $a_{1}a_{2}=01$, $b_{1}b_{2}=10$, and $\beta>1$, for the lexicographic or alternating lexicographic order on $\{0,1\}^{N}$. Set
$$ K(z)=\sum_{i=1}^{\infty}(b_{i}-a_{i})z^{i-1}=\langle b\rangle_{1/z}-\langle a\rangle_{1/z} $$ In case of the lexicographic order, $1/\beta$ is the smallest positive root of $K(z)$. In the alternating case, $-1/\beta$ is the largest negative root of $K(z)$.
\end{lemma}

\par According to the Lemma 6, we can calculate the $\beta$ of admissible kneading invariants. What follows we will put forward an expression of $\alpha$ through the iteration of any given $x\in[0,1]$.
\begin{theorem}
Let $T_{\beta,\alpha}(x)=\beta x+\alpha$ mod 1 be a linear mod 1 transformation, $1<\beta\leq2$ and $0\leq\alpha\leq2-\beta$. $\forall x\in[0,1]$, denote $k(x)=(x_{1},x_{2}\cdots x_{n},\cdots)$ as the kneading sequence of $x$, then $$\alpha=(\beta-1)(\sum_{i=1}^{\infty}\frac{x_{i}}{\beta^{i}}-x).$$
%In particular, if $x=0$, $\alpha=(\beta-1)\sum_{i=1}^{\infty}\frac{0_{i}}{\beta^{i}},$ if $x=1$,$\alpha=(\beta-1)(\sum_{i=1}^{\infty}\frac{1_{i}}{\beta^{i}}-1)$.

\begin{proof}\rm
$\forall x\in[0,1]$, we have
 $$ T(x)=\beta x+\alpha-x_{1},T^{2}(x)=\beta T(x)+\alpha-x_{2}=\beta^{2}x+(\beta+1)\alpha-(\beta x_{1}+x_{2})$$
 By induction, we have $$T^{n}(x)=\beta^{n}x+(\beta^{n-1}+\cdots+\beta+1)\alpha-\beta^{n}(\frac{x_{1}}{\beta}+\frac{x_{2}}{\beta^{2}}+\cdots+\frac{x_{n}}{\beta^{n}})
 $$
i.e.,$$T^{n}(x)=\beta^{n}x+\alpha\beta^{n}\sum_{i=1}^{n}\frac{1}{\beta^{i}}-\beta^{n}\sum_{i=1}^{n}\frac{x_{i}}{\beta^{i}}
 $$
 Since $[0,1]$ is compact interval, $\{T^{n}(x)\}_{n=1}^{\infty}$ admits accumulation point. There are convergent subsequence $\{T^{n_{k}}(x)\}$ such that $T^{n_{k+1}}(x)-T^{n_{k}}(x)=\varepsilon_{k}(x)\rightarrow0$ and $n_{k+1}-n_{k}\rightarrow\infty $ as $k\rightarrow\infty$. Then we can have:
 \[\varepsilon_{k}(x):=T^{n_{k+1}}(x)-T^{n_{k}}(x)\]
\[=(\beta^{n_{k+1}}x+\alpha\beta^{n_{k+1}}\sum_{i=1}^{n_{k+1}}\frac{1}{\beta^{i}}-\beta^{n_{k+1}}\sum_{i=1}^{n_{k+1}}\frac{x_{i}}{\beta^{i}})
-(\beta^{n_{k}}x+\alpha\beta^{n_{k}}\sum_{i=1}^{n_{k}}\frac{1}{\beta^{i}}-\beta^{n_{k}}\sum_{i=1}^{n_{k}}\frac{x_{i}}{\beta^{i}})\]
\[=(\beta^{n_{k+1}}-\beta^{n_{k}})x+\alpha(\beta^{n_{k+1}-1}+\cdots+\cdots\beta^{n_{k}})
-(\beta^{n_{k+1}}\sum_{i=1}^{n_{k+1}}\frac{x_{i}}{\beta^{i}}-\beta^{n_{k}}\sum_{i=1}^{n_{k}}\frac{x_{i}}{\beta^{i}})\]
 \[=(\beta^{n_{k+1}}-\beta^{n_{k}})x+\alpha\beta^{n_{k}}(1+\beta+\cdots+\beta^{n_{k+1}-n_{k}-1})\]
 \[-(\beta^{n_{k+1}}-\beta^{n_{k}})\sum_{i=1}^{n_{k}}\frac{x_{i}}{\beta^{i}}-\beta^{n_{k+1}}\sum_{i=n_{k}+1}^{n_{k+1}}\frac{x_{i}}{\beta^{i}}\]
Divided $(\beta^{n_{k+1}}-\beta^{n_{k}})$ on both sides of the equality, we obtain
$$\frac{\varepsilon_{k}(x)}{\beta^{n_{k+1}}-\beta^{n_{k}}}=x+\frac{\alpha}{1-\beta}-\sum_{i=1}^{n_{k}}\frac{x_{i}}{\beta^{i}}
-\frac{1}{1-\beta^{n_{k}-n_{k+1}}}\sum_{i=1}^{n_{k}}\frac{x_{i}}{\beta^{i}}
$$
Then let $k\rightarrow\infty$, we can have:
$$0=x+\frac{\alpha}{1-\beta}-\sum_{i=1}^{\infty}\frac{x_{i}}{\beta^{i}},$$
which is equivalent to
$$\alpha=(\beta-1)(\sum_{i=1}^{\infty}\frac{x_{i}}{\beta^{i}}-x).
$$
%Take $x=0$ and $x=1$ in the above expression, one obtain the desired results.
\end{proof}
\end{theorem}
The proof is completed. $\hfill\square$
\begin{remark}\rm
\
\begin{enumerate}[(1)]
%\item If we know the some periodic point
%\item If  $k(x)=k_{+}$ or $k_{-}$ in Theorem 2, we can have $$\alpha=1-\beta+\beta(\beta-1)\sum_{i=1}^{\infty}\frac{k(x)}{\beta^{i}}$$.
\item Since $k(0)=\sigma(k_+)$, by Theorem 2, $\alpha$ can be calculated in terms of $k_+$ and $\beta$:  $$ \alpha=(\beta-1)\sum_{i=1}^{\infty}\frac{k_i(0)}{\beta^{i}},$$ where $k_{i}(0)$ is the $i$th digit of $k(0)=\sigma(k_{+})$.
\item If $k_+=10^{\infty}$, i.e., $k(0)=0^{\infty}$ is a fixed point, then  $\alpha=(\beta-1)\sum_{i=1}^{\infty}\frac{k_i(0)}{\beta^{i}}=0$.
\item Since $k(1)=\sigma(k_-)$,  $\alpha$ can be calculated in terms of $k_-$ and $\beta$:  $$ \alpha=(\beta-1)(\sum_{i=1}^{\infty}\frac{k_i(1)}{\beta^{i}}-1),$$ where $k_{i}(1)$ is the $i$th digit of $k(1)=\sigma(k_{-})$.
%\item There is a one-to-one corresponding between prime expansive kneading invariants and $(\beta,\alpha)\in \Delta=\{(\beta, \alpha): 1 < \beta \le 2, 0 \le \alpha \le 2-\beta\}$.

\item If $k_-=01^{\infty}$, i.e., $k(1)=1^{\infty}$ is a fixed point, then  $\alpha=(\beta-1)(\sum_{i=1}^{\infty}\frac{1}{\beta^{i}}-1)=(\beta-1) ( \frac{1}{\beta}  \frac{1}{1-\frac{1}{\beta}}-1)=2-\beta$.
\end{enumerate}
\end{remark}

\subsection{Examples}
\noindent Let $K_f=(k_+, k_-)$ be an expansive  prime kneading invariant, it is conjugate to a linear mod 1 transformation $T_{\beta, \alpha}(x)=\beta x + \alpha$ $mod \ 1$. So there is a point $(\beta, \alpha)$ in the parametric space corresponds to $K_f$. We have known a lot about $\beta$, the topological entropy $h(f)$ of $f$ equals $\ln \beta$, $\ln \beta$ is the exponential growth rate of the laps number $I(f^{n})$ or the complexity function $C(n)$, and accoding to Lemma 6, $\beta$ equals to $1/z$ where $z$ is the smallest positive root of $K(z)$. But we know very little about $\alpha$. If $\beta=1$, $\alpha$ is the rotation number of $f$. In the factorization theorem, most of the factors are periodic kneading invariants, which are also subshift of finite types.
%If We do not know if $\alpha$ may So is $\alpha$ closely related to rotation number? What follows we will show the essence of $\alpha$. Is $\alpha$ up to the lap numbers?
\begin{example}
 \rm Let $K_1=(k_{+},k_{-})=((10)^{\infty},(011)^{\infty})$. According to Lemma 6, we can obtain
\begin{center}
         {$K_1(z)=(1-z-z^5)(1+z^6+z^{12}+z^{18}+z^{24}+\cdots)=(1-z-z^5)/(1-z^6)$. }
\end{center}The power series $K(z)$ has the smallest positive zero $z=1/\beta \approx 0.7548$ in $(0,1)$, $\beta \approx 1.3247$. On the other hand, the kneading invariant corresponds a subshift of finite type,  the transition matrix is
\begin{center}
$A_{1}=\left(
\begin{array}{ccc}
0 & 1 & 1 \\
1 & 0 & 0 \\
0 & 1 & 0 \\
\end{array}
\right).$
\end{center}
We calculate the eigenvalues of the matrix are approximately $1.3247$, and $-0.66 \pm 0.56 i$. It is obvious that the largest eigenvalue equals to $\beta$.
Since $k(0)=(01)^{\infty}$, according to  Theorem 2, we know that:
\begin{center}
{$ \alpha_1=\frac{\beta-0}{\beta+\beta^{2}}=\frac{1}{1+\beta}\approx 0.4301597$.}

\end{center}

\end{example}
\par Note that the transition matrix of an admissible prime periodic kneading invariant corresponds to $(\beta, \alpha)$. The eigenvalues and eigenvectors capture the essential properties of a matrix. $\beta$  is the spectral radius of the transition matrix, it is interesting to ask wether the other two eigenvalues may relate to $\alpha$.

\begin{example}\rm
Let $K_2=(k_{+},k_{-})=((100)^{\infty},(01)^{\infty})$. We can observe $K_2$ is the dual of $K_1$, i.e.,  $K_2$  is obtained  via interchange  $0$ and $1$ in $K_1$.  By Lemma 6, we get
	\begin{center}
		{$K(z)=(1-z-z^5)/(1-z^6)$, }
	\end{center}
which is the same as for $K_1(z)$.  So they have the same zeros, and the same entropy $\ln\beta$. As a subshift of finite type, the transition matrix corresponds to $K_2$ is
\begin{center}
$A_{2}=\left(
\begin{array}{ccc}
0 & 1 & 0 \\
0 & 0 & 1 \\
1 & 1 & 0 \\
\end{array}
\right)$.
\end{center}
	
	The eigenvalues of $A_{2}$ are the same as $A_1$, and the largest eigenvalue equals to $\beta$.
	Since $k(0)=(001)^{\infty}$, using Theorem 2, we obtain:
	\begin{center}
		{$ \alpha_2= \frac{\beta}{\beta+\beta^{2}+\beta^{3}}\approx 0.245122$. }
		
	\end{center}

\end{example}
\
As we can see from the examples, for two (dual) admissible expansive periodic kneading invariants $K_1$ and $K_2$, we have $K_1(z)=K_2(z)$, and the transition matrixes admit the same eigenvalues. So they  have the same entropy, but different intercepts:$\alpha_1 \approx 0.4301597$ and $\alpha_2 \approx 0.245122$. So $\alpha$ is not determined by the eigenvalues of transition matrix. Remember that when $\beta=1$, the topological entropy of $T_{\beta, \alpha}$ is zero, $T_{\beta, \alpha}$ is conjugated to a circle homeomorphism, and $\alpha$ is the rotation number of the transformation. Notice that the rotation mumber of an orbit $orb(x)$ equals to the frequency of $1$ in $k(x)$ [3].
 %The following Corollary indicates that $\alpha$ may be regarded as a generalized rotation number.

We denote by $e:
\mathbb{R} \to \mathbb{S}^1=\{z\in \mathbb{C} :|z|=1\}$ the natural
covering map $e(x)=\exp(2\pi ix)$. Let $f$ be a Lorenz map, not
necessarily expansive. There exists a map $F: \mathbb{R} \to
\mathbb{R}$ such that $e\circ F=f \circ e$ and $F(x+1)=F(x)+1$. $F$
is called a degree one lifting of $f$. Furthermore, if $F(0)=f(0)$,
then there exists a unique such lifting.

The {\em rotation number} of $f$ at $x$ is defined by
$$
\rho(x)= \lim_{n \to \infty} \frac{F^n(x)-x}{n}
$$
if the limit exists.
It is known that the set of all rotation numbers $\rho(x)$ of $f$ is
an interval and that this interval is reduced to a singleton when
$f(0)=f(1)$ . The rotation number is tightly relate to
the number of returns of $x$ into the interval $(c,\ 1]$, defined by
\begin{equation}\label{Returnnumber}
m_n(x)= \# \left\{0\le i <n: f^i(x) \in (c,\ 1] \right\},
\end{equation} and
$$
\rho(x)= \lim_{n \to \infty} \frac{m_n(x)}{n},
$$
if the limit exists.

A sequence ${\bf a}=a_1a_2a_3\cdots \in \Omega_2$ is said to be rotational if there exists $\rho \in [0, 1]$ such that ${\bf a}$ is the kneading sequence of $0$ under the rotation $R_{\rho}(x)=x+\rho \ \ mod\ 1$.  If ${\bf a}$ is rotational, it is  shift minimal in the sense  $\sigma^n({\bf a}) \ge {\bf a}$ for all positive integer $n$. Fix a rotation sequence ${\bf a}$, let $B_{{\bf a}}$ be the collection of $\beta$ such that there is linear mod one transformation $T_{\beta, \alpha}$ such that ${\bf a}$ is the kneading sequence of $0$. For each $\beta \in  B_{{\bf a}}$,  by Remark after Theorem *, we can obtain  $\alpha=(\beta-1)\sum_{i=1}^{\infty} \frac{a_i}{\beta^i}$. The following  Proposition tell us that when $\beta \searrow 1$, $\alpha$ will approaches to the rotation number of ${\bf a}$.
\begin{proposition} Let ${\bf a}$ be a rotation sequence with rotation number $\rho$. For each $\beta \in  B_{{\bf a}}$, we have
\begin{equation}\label{alpha}
\alpha: =\alpha(\beta)=(\beta-1)\sum_{i=1}^{\infty} \frac{a_i}{\beta^i}, \ \ and \ \ \  \lim_{\beta \searrow 1} \alpha(\beta)=\rho.
\end{equation}

\begin{proof}\rm
%From the Remark of Theorem 2, we know that $\alpha: =\alpha(\beta)=(\beta-1)\sum_{i=1}^{\infty} \frac{a_i}{\beta^i}$.
Since ${\bf a}$ is a rotation sequence with rotation number $\rho$, ${\bf a}$ is the kneading sequence of $0$ under the action of the rotation $R_{\rho}(x)=x + \rho$ (mod 1), ${\bf a}$ is shift minimal, and the rotation number of ${\bf a}$ is $\rho$, i.e. $\lim_{n \to \infty}\frac{\sum_{i=1}^{n}a_{i}}{n}$. Fix $\beta$ in  $B_{{\bf a}}$, take $x=0$ in  Theorem 2, we know that $\alpha: =\alpha(\beta)=(\beta-1)\sum_{i=1}^{\infty} \frac{a_i}{\beta^i}$.

Let $u(x)=\sum_{n=0}^{\infty}a_{n+1}x^{n}$, then $u(x)=\frac{1-x}{x}\sum_{n=1}^{\infty}a_{n}x^{n}$ and $u(\frac{1}{\beta})=\alpha(\beta)$. In what follows we show that $\lim_{x\rightarrow1^{-}}u(x)=\rho$, which implies $\lim_{\beta\rightarrow1^{+}}\alpha(\beta)=\rho$.

\noindent Since $a_{n}=1$ or $a_{n}=0$, we know that the radius of convergence of the power series $u(x)=(1-x)\sum_{n=1}^{\infty}a_{n}x^{n-1}=\sum_{n=1}^{\infty}(a_{n+1}-a_{n})x^{n}$ is 1, so $u(x)=(1-x)\sum_{n=1}^{\infty}a_{n+1}x^{n}$ converges uniformly in $[0,1)$.

\noindent Observe that for $x\in[0,1)$, $\frac{1}{1-x}=\sum_{n=0}^{\infty}x^{n}$, then we can have
\begin{equation}
\frac{1}{(1-x)^{2}}=\sum_{n=0}^{\infty}(n+1)x^{n}.
\end{equation}
By direct calculation we obtain $\frac{u(x)}{(1-x)^{2}}=(1+x+x^{2}+\cdots)(a_{1}+a_{2}x+a_{3}x^{2}+\cdots)$
\[=a_{1}+(a_{1}+a_{2})x+(a_{1}+a_{2}+a_{3})x^{2}+\cdots\]
\[=\sum_{n=0}^{\infty}S_{n+1}x^{n}, \]
where $S_{n+1}=\sum_{i=1}^{n+1}a_{i}$. It follows that $u(x)=(1-x)^{2}\sum_{n=0}^{\infty}S_{n+1}x^{n}$ for $x\in[0,1)$. Since the rotation number of ${\bf a}$ is $\rho$, we know that $\lim_{n\rightarrow\infty}\frac{S_{n+1}}{n+1}=\rho$.

\noindent $\forall\varepsilon>0$, $\exists N>0$, when $n>N$, $|\frac{S_{n+1}}{n+1}-\rho|<\frac{\varepsilon}{2}$, then
$$
|u(x)-\rho|=|(1-x)^{2}\sum_{n=0}^{\infty}S_{n+1}x^{n}-\rho|\leq|(1-x)^{2}\sum_{n=0}^{\infty}S_{n+1}x^{n}|+
|(1-x)^{2}\sum_{n=N+1}^{\infty}S_{n+1}x^{n}-\rho|
$$
Using (3.11), we have
\[|(1-x)^{2}\sum_{n=N+1}^{\infty}S_{n+1}x^{n}-\rho|=
|(1-x)^{2}\sum_{n=N+1}^{\infty}\frac{S_{n+1}}{n+1}(n+1)x^{n}-\rho|\]
\[=|\frac{\sum_{n=N+1}^{\infty}\frac{S_{n+1}}{n+1}(n+1)x^{n}-\rho\sum_{n=0}^{\infty}(n+1)x^{n}}
{\sum_{n=0}^{\infty}(n+1)x^{n}}|\]
\[=|\frac{\sum_{n=N+1}^{\infty}(\frac{S_{n+1}}{n+1}-\rho)(n+1)x^{n}-\rho\sum_{n=0}^{N}(n+1)x^{n}}
{\sum_{n=0}^{\infty}(n+1)x^{n}}|\]
\[<|\frac{\frac{\varepsilon}{2}\sum_{n=N+1}^{\infty}(n+1)x^{n}}{\sum_{n=0}^{\infty}(n+1)x^{n}}|
+\rho(1-x)^{2}\sum_{n=0}^{N}(n+1)x^{n}\]
\[<\frac{\varepsilon}{2}+\rho(1-x)^{2}\sum_{n=0}^{N}(n+1)x^{n}\]
Combine the equations above, we can obtain
\[|u(x)-\rho|<\frac{\varepsilon}{2}+(1+\rho)(1-x)^{2}\sum_{n=0}^{N}(n+1)x^{2}\]
\[<\frac{\varepsilon}{2}+(1+\rho)(1-x)\frac{(N+1)(N+2)}{2}\]
Choose $\delta=|1-x|<\frac{\epsilon}{(1+\rho)(N+1)(N+2)}$, then we get $|u(x)-\rho|<\varepsilon$, which indicates $\lim_{x \to 1^{-}}u(x)=\rho$, the proof is completed. $\hfill\square$
\end{proof}
\end{proposition}

\begin{remark}\rm
\
\begin{enumerate}[(1)]
\item If $\rho =\frac{k}{m}$ is rational, then $R_{\rho}(x)=x+\rho \ \ mod\ 1$ is a rational rotation, ${\bf a}$ is $m-$periodic, we can get $\lim_{\beta \searrow 1} \alpha(\beta)=\rho$ by direct calculation. In fact,
 $\alpha(\beta)=(\beta-1)(\frac{a_1}{\beta}+ \frac{a_2}{\beta^2}+\cdots+ \frac{a_m}{\beta^m})(1+\frac{1}{\beta^m} +\frac{1}{\beta^{2m}}+\frac{1}{\beta^{3m}}+\cdots)$
\[=(\beta-1)(\frac{a_1}{\beta}+ \frac{a_2}{\beta^2}+\cdots+ \frac{a_m}{\beta^m})\frac{1}{1-\beta^{-m}}\]
\[= (\beta-1)(\frac{a_1}{\beta}+ \frac{a_2}{\beta^2}+\cdots+ \frac{a_m}{\beta^m}) \frac{\beta^m}{(\beta-1)(1+\beta+\cdots+\beta^{m-1})}\]
\[= (\frac{a_1}{\beta}+ \frac{a_2}{\beta^2}+\cdots+ \frac{a_m}{\beta^m})\frac{\beta^m}{1+\beta+\cdots+\beta^{m-1}}.\]
It follows that $\lim_{\beta\rightarrow1^{+}}\alpha(\beta)=\frac{k}{m}=\rho$.
\item If ${\bf a}$ is not rotational, then we can not let $\beta \to 1$ because it may violate the Hubbard-Sparrow expansive condition.
\end{enumerate}
\end{remark}

\section{Classification of expansive Lorenz maps}
For convenience, we denote $L_{r}$ be the collection of expansive renormalizable Lorenz maps. In 1990, John H. Hubbard and Colin T. Sparrow [9] had put forward the classification of topologically expansive Lorenz maps. They showed that topologically expansive Lorenz maps can be described up to topological conjugacy by their kneading invariants. Our Theorem gives an extension of [9], we can not only justify whether two expansive Lorenz maps are topological conjugate or not, but also describe the distance between two expansive Lorenz maps which are not conjugate.
\begin{lemma}{\bf([9,Theorem 1])}
An expansive Lorenz map $f$ is conjugate to a $\beta$-transformation if and only if $f$ is finitely renormalizable and each renormalization of $f$ is periodic.
\end{lemma}
 As for $f\in L_{r}$, periodic renormalization is relevent to the conjugacy problem. Glendinning[9] showed that an expansive Lorenz map is conjugate to a $\beta$-transformation if its renormalizations are all periodic renormalizations.
\begin{proposition}\label{prime}
Let $f$ be a renormalizable expansive Lorenz map with minimal CICS $D$. We have:
\begin{enumerate}[(1)]
\item The map $\phi_{1}:=f|_{D}:D\rightarrow D$ which means $f$ restricted on $D$ is prime.
\item  If $D$ is a Cantor set, then restriction of $f$ on $D$ corresponds to a pair $(\beta_{1},\alpha_{1})$ such that $\phi_{1}:=f|_{D}:D\rightarrow D$ is conjugated to $T_{\beta_{1},\alpha_{1}}$.
\item If $D=O$ where $O$ is the periodic orbit of minimal period. Then $\psi_{1}:=f|_{D}:D\rightarrow D$ corresponds to a pair $(1,\rho(f))$.
    \end{enumerate}
\end{proposition}
\begin{proof}\rm
(1)According to the Definition 1 and Theorem 1, $D=\Sigma_{w_{+}^{\infty},w_{-}^{\infty}}$ where $w_{+}$ and $w_{-}$ are finite words with total length longer than $1$. Suppose $D$ is renormalizable, then we can find $w_{+}^{'}$ and $w_{-}^{'}$ within $w_{+}^{\infty}$ and $w_{-}^{\infty}$, which is contradict to the definition of renormalization. Hence the map $f$ restricted on $D$ is prime.

\noindent(2)If $D$ is a Cantor set, then $D$ has infinite words. We have just proved $\phi_{1}:=f|_{D}:D\rightarrow D$ is prime, according to Remark 3, there is a one-to-one corresponding between prime expansive Lorenz map and $(\beta,\alpha)$. So $\phi_{1}$ is conjugate to $T_{\beta_{1},\alpha_{1}}$.

\noindent(3) If $D=O$ where $O$ is the periodic orbit of minimal period, then $D$ have only finite words, $w_{+}^{\infty}$ and $w_{-} ^{\infty}$ are on the same periodic orbit. So the entropy $h(\psi_{1})=0$, which means $\ln(\beta)=0$, and $\psi_{1}:=f|_{D}:D\rightarrow D$ corresponds to a pair $(1,\rho(f))$. $\hfill\square$

\end{proof}

{\bf Main Theorem.} \emph {Fix $0\leq m\leq\infty$. There is a one-to-one correspondence between $\mathcal{L}_m$ and $\mathcal{A}_m$. More precisely, for $f \in \mathcal{L}_m$, there is a sequence $A\in \mathcal{A}_m$ associated with $f$, and for each sequence $B \in \mathcal{A}_m $, there exists an $m$-renormalizable expansive Lorenz map $g$ associated with $B$. Two expansive Lorenz maps are topologically conjugated if and only if they associated with the same sequence. Moreover, the cluster of points in the sequence are complete topological invariants for expansive Lorenz maps.}

Now we are in a position to prove the main Theorem.

\begin{proof}\rm
Let $f$ be an $m$-renormalizable expansive Lorenz map with kneading invariant $K=(k_+, k_-)$, we try to find a sequence in $\mathcal{A}_m$ associated with it. If $m<\infty$, by the factorization Theorem, there are $m$  pair or finite words $W_i=(w_{i,+}, w_{i,-})$ so that $W_i^{\infty}=(w_{i,+}^{\infty}, w_{i,-}^{\infty})$ ($i=1, \cdots, m$)is  periodic prime kneading invariant such that
 $$K=W_1*\cdots*W_m*K^*,$$
 where $K^*$ is the kneading invariant of $R^m K$, which is admissible and prime because $f$ is $m$-renormalizable. By  Hubbard-Sparrow theorem, there is an expansive  and prime Lorenz map $g_*$ such that $K^*$ is the kneading invariant of $g_*$.
 According to Parry's theorem, $g_*$ is conjugated to a prime linear mod one transformation $T_{\beta_*, \alpha_*}$. Hence, the last factor $K$ is associated with a point $(\beta_*, \alpha_*) \in \Delta_{pe}$. For $i=1, \ldots, m$, since $W_i=(w_{i,+}, w_{i,-})$ is periodic prime kneading invariant,  $w_{i,+}^{\infty}$ and $w_{i,-}^{\infty}$ are  periodic points of $R^{i-1}f$. If they belong to the same periodic orbit, then the $i$-th renormalization of $f$ is periodic, the minimal completely invariant closed set of $R^{i-1}f$ is just a periodic orbit. The restriction of  $R^{i-1}f$ on the set is a rational rotation with rational rotation number $\alpha_i$. The factor $W_i$ is associated with $(1, \alpha_i)\in \Delta_r \subset \Delta_{pp}$. If $w_{i,+}^{\infty}$ and $w_{i,-}^{\infty}$ belong to different periodic orbits, then the $i$-th renormalization of $f$ is not periodic. It follows that $W_i^{\infty}=(w_{i,+}^{\infty}, w_{i,-}^{\infty})$ satisfies the Hubbard-Sparrow condition, which implies that there is an expansive Lorenz map $g_i$ whose kneading invariant is $W_i^{\infty}=(w_{i,+}^{\infty}, w_{i,-}^{\infty})$. Remember that $W_i^{\infty}=(w_{i,+}^{\infty}, w_{i,-}^{\infty})$ is prime, $g_i$ is conjugated to a linear mod one transformation $T_{\beta_i, \alpha_i}$.  $w_{i,+}^{\infty}$ and $w_{i,-}^{\infty}$ are  periodic points indicates $T_{\beta_i, \alpha_i}$ is periodic. So the prime and periodic kneading invariant  $W_i^{\infty}=(w_{i,+}^{\infty}, w_{i,-}^{\infty})$  is associated with $(\beta_i, \alpha_i) \in \Delta_{pp}$.

 If $m=\infty$, by the same arguments in the case $m<\infty$, for the factor $W_i$ $(i=1, 2, \ldots)$, there is a point $(\beta_i, \alpha_i) \in \Delta_{pp}$ associated with  $W_i$. $\hfill\square$
\end{proof}

\begin{theorem}
Let $f$ be an $m(0\leq m\leq\infty)$ times  renormalizable expansive Lorenz map, then we have
\begin{enumerate}[(1)]
    \item  If $m<+\infty$, there are a cluster of $m+1$ points $\{A_1, A_2, \cdots, A_m, A_*\} \subset \Delta$  associated with $f$. Moreover, $A_i=(\beta_i, \alpha_i) \in \Delta_{pp} \ (i=0,1,2\ldots m)$, and $A_*=(\beta_*, \alpha_*) \in \Delta_{pe}$.

    \item If $m=\infty$, there are a cluster of infinite points $\{A_1, A_2, A_3, \cdots\} \subset \Delta_{pp}$ associated with $f$.

    \item Let $\{A_i\}_{i=1}^m$ be a cluster of points in $\Delta_{pp}$, and $A_* \in \Delta_{pe}$ when $m<\infty$, then there is an expansive Lorenz map associated

    \item Two expansive Lorenz maps are topologically conjugate if and only if they admit the same cluster of points in parametric space,  and the cluster of points are completed topological invariants for expansive Lorenz maps.
\end{enumerate}
\begin{proof}\rm We know that an expansive Lorenz map can be prime or renormalizable depends on whether the kneading invariants satisfy (2.5) or not. So we prove this Theorem from two cases.
\

\noindent{\bf Case 1:} We first consider $m=0$ which means the expansive Lorenz map $f$ is prime. According to Remark 3, if $f$ is prime, then $f$ is topologically conjugate to a $\beta$-transformation. Hence we have the following:
\begin{center}
$f$ is prime $\Leftrightarrow$ $T_{\beta,\alpha}(x)$ $\Leftrightarrow$ pair $(\beta,\alpha)$
\end{center}
As a result, we can use $(\beta,\alpha)$ to represent an expansive prime Lorenz map, and given two expansive prime Lorenz maps, they are conjugate if and only if they have the same pair $(\beta,\alpha)$.
\

\noindent{\bf Case 2:} Now we will consider the renormalizable Lorenz maps. For any given $f\in L_{r}$ with $m(0<m\leq\infty)$ times renormalization, we will show there exists unique sequence of pairs $(\beta_{i},\alpha_{i})(i=1,2\ldots m+1)$.

\  According to the Remark 3, renormalization is essentially a reduction mechanism. If
$(\Sigma_{k_{+},k_{-}},\sigma)$ is renormalizable with finite
words $w_+$ and $w_-$, then its dynamics is captured by two Lorenz
maps: the renormalization map $(\Sigma_{k_{+}^{1},k_{-}^{1}},
\sigma)$ and the dynamics on the associated set
$(\Sigma_{w_+^{\infty}, w_-^{\infty}},\sigma)$. In other words, after once renormalization, the dynamical behavior of $f$ can be captured by two Lorenz maps, one is restricted in $[a_{1},b_{1}]=[\sigma^{r}(k_{+}),\sigma^{l}(k_{-})]$ which we called $Rf$, and the other is restricted in $D_{1}=\{\delta\in\Sigma_{k_{+},k_{-}},orb(\delta)\cap(\sigma^{r}(k_{+}),\sigma^{l}(k_{-}))=\emptyset\}$ which we called $\phi_{1}$. According to the Proposition 2 and case 1, $\phi_{1}$ is prime and there is a corresponding pair $(\beta_{1},\alpha_{1})$. Moreover, if $Rf$ is prime on $[a_{1},b_{1}]$, then $Rf$ also conjugate to a $\beta$-transformation $(\beta_{2},\alpha_{2})$. In this condition, $m=1$, $f\in L_{r}$ can be only once renormalizable, $f$ corresponds to a unique sequence $(\beta_{1},\alpha_{1})$ and $(\beta_{2},\alpha_{2})$.
 \
\par However, if $Rf$ is renormalizable, we repeat the proceed above. Denote $R^{2}f$ be the minimal renormalization of $Rf$ and $[a_{2},b_{2}]$ be the renormalization interval, which means the dynamical behavior of $Rf$ can be captured by two Lorenz maps: the renormalization map $(\Sigma_{k_{+}^{2},k_{-}^{2}},\sigma)$ and and the dynamics on the associated set
$(\Sigma_{w_{+,1}^{\infty}, w_{-,1}^{\infty}},\sigma)$. So $Rf$ restricted on $D_{2}=\Sigma_{w_{+,1}^{\infty}, w_{-,1}^{\infty}}$ is prime and corresponds to a pair $(\beta_{2},\alpha_{2})$. In this way, if $f\in L_{r}$ is $m(0<m\leq\infty)$ times renormalizable, we can obtain a unique sequence of pairs $(\beta_{i},\alpha_{i})$, $i=1,2\ldots m+1$. $\hfill\square$
\end{proof}

\end{theorem}

\begin{definition}\rm
Let $d(f,g)$ be the distance of two expansive Lorenz maps $f$ and $g$, denote $p$ be the maximal integer such that the first $p$ pairs of $f$ and $g$ are identical, then we can obtain:
$$
d(f,g)=\frac{1}{2^{p}}(1+\frac{1}{2^{s_{+}}}+\frac{1}{2^{s_{-}}})
$$
where $s_{+}$ means the first $s_{+}$ words of $k_{f+}^{p+1}$ and $k_{g+}^{p+1 }$ are identical, $s_{-}$ means the first $s_{-}$ words of $k_{f-}^{p+1}$ and $k_{g-}^{p+1 }$ are the same.
\end{definition}
\begin{lemma}
The distance $d(f,g)$ is a new metric and it admits triangle inequality.
\end{lemma}
\begin{proof} \rm For any given expansive Lorenz maps $f,g,h$, we have $d(f,g)=\frac{1}{2^{p_{1}}}(1+\frac{1}{2^{s_{1}}}+\frac{1}{2^{t_{1}}})$,  $d(g,h)=\frac{1}{2^{p_{2}}}(1+\frac{1}{2^{s_{2}}}+\frac{1}{2^{t_{2}}})$. If $p_{1}= p_{2}$, it is obvious that triangle inequality admits. Now we suppose $p_{1}< p_{2}$, then $d(f,h)=\frac{1}{2^{p_{1}}}(1+\frac{1}{2^{s_{3}}}+\frac{1}{2^{t_{3}}})$ and we can have the following:
\begin{equation}
\left \{
\begin{array}{ll}
d(f,g)=\frac{1}{2^{p_{1}}}(1+\frac{1}{2^{s_{1}}}+\frac{1}{2^{t_{1}}}) \\
d(f,h)=\frac{1}{2^{p_{1}}}(1+\frac{1}{2^{s_{3}}}+\frac{1}{2^{t_{3}}}) \\
d(g,h)=\frac{1}{2^{p_{2}}}(1+\frac{1}{2^{s_{2}}}+\frac{1}{2^{t_{2}}})
\end{array}
\right.
\end{equation}
Since $p_{1}< p_{2}$, we can get:
\begin{equation*}
\left \{
\begin{array}{ll}
2^{p_{1}}d(f,g)= (1+\frac{1}{2^{s_{1}}}+\frac{1}{2^{t_{1}}}) \\
2^{p_{1}}d(f,h)= (1+\frac{1}{2^{s_{3}}}+\frac{1}{2^{t_{3}}}) \\
2^{p_{1}}d(g,h)=\frac{1}{2^{p_{2}-p_{1}}}(1+\frac{1}{2^{s_{2}}}+\frac{1}{2^{t_{2}}})
\end{array}
\right.
\end{equation*}
Let $p=p_{2}-p_{1}>0$, then
\begin{equation}
\left \{
\begin{array}{ll}
2^{p_{1}}d(f,g)= (1+\frac{1}{2^{s_{1}}}+\frac{1}{2^{t_{1}}})=M \\
2^{p_{1}}d(f,h)= (1+\frac{1}{2^{s_{3}}}+\frac{1}{2^{t_{3}}})=N \\
2^{p_{1}}d(g,h)=\frac{1}{2^{p}}(1+\frac{1}{2^{s_{2}}}+\frac{1}{2^{t_{2}}})=K
\end{array}
\right.
\end{equation}
It is clear that $M+N>K$, so we will verify $M+K>N$ and $N+K>M$. To the straight sense, we may assume that triangle inequality does not exist because when $p$ is big enough, $1/2^{p}$ may be even smaller than $1/2^{q}$ where $q=\min\{s_{1},s_{3},t_{1},t_{3}\}$. In fact, with the increase of $p$, $q$ is also increasing and $p<q$ always exists. According to the equation (4.9) and (4.10), we can see that for any given three expansive Lorenz maps, if $p_{1},p_{2}>0$, we can transform them into one prime map and two renormalizable maps. And (4.10) indicates $f'$ is prime, $g',h'$ are renormalizable with first $p$ pairs are identical. Suppose $(k_{g'+}^{p+1},k_{g'-}^{p+1}=10\cdots,01\cdots)$ and $(k_{h'+}^{p+1},k_{h'-}^{p+1}=10\cdots,01\cdots)$, we have known that $w_{+}=10,w_{-}=01$ is the shortest renormalization words. So before $p$ times renormalization, $k_{g'+}$ and $k_{h'+}$ at least have $m\geq2^{p+1}$ same words, $k_{g'-}$ and $k_{h'-}$ at least have $n\geq2^{p+1}$ same words. It is clear that for any $p>0$, $p<2^{p+1}$ always exsits. So we can obtain $1/2^{p}>1/2^{2^{p+1}}$. Now we consider case 1: $s_{1}<m$, $t_{1}<n$, in this case, $s_{1}=s_{3}$, $t_{1}=t_{3}$, and $M=N$, the triangle inequality admits. Case 2: $s_{1}>m$, $t_{1}>n$, in this case, $s_{3}\geq m$ and $t_{3}\geq n$, $1/2^{p}>\max\{1/2^{t_{1}},1/2^{t_{3}},2^{s_{1}},1/2^{s_{3}}\}$, this means $M+K>N$ and $N+K>M$. It is similar with  $s_{1}<m$, $t_{1}>n$ or $s_{1}>m$, $t_{1}<n$. $\hfill\square$
\end{proof}
\begin{example}\rm Let
$$\left \{
\begin{array}{ll}
k_{g+}= (100110110111001110100110110111010011011)^{\infty},\\
k_{g-}= (01110100110111001101101110)^{\infty}.
\end{array}
\right.
$$
$$\left \{
\begin{array}{ll}
k_{h+}= (10011011011100111010011011)^{\infty} ,\\
k_{h-}= (011101001101110011011011101001101101110)^{\infty}.
\end{array}
\right.
$$
We can see that $g$ and $h$ have three identical pairs of $(\alpha,\beta)$ because they have the same renormalizable words : $(w_{1+},w_{1-})=(10,011)$, $(w_{2+},w_{2-})=(100,01)$, $(w_{3+},w_{3-})=(10 ,01)$. So $p=3$ and $k_{g+},k_{h+}$ have 26$>2^{3+1}$ same words, $k_{g-},k_{h-}$ have 26 $>2^{3+1}$ same words. This indicates for any given prime kneading invariants $(k_{f+},k_{f-})$, there are only two cases: either $M=N$ or $1/2^{3}>\max\{1/2^{t_{1}},1/2^{t_{3}},2^{s_{1}},1/2^{s_{3}}\}$. Both two cases can lead to the triangle inequality.
\end{example}
\begin{remark}\rm Let $\mu,\nu$ be two sequences in $\{0,1\}^{N}$, the classical metric is defined as $d(\mu,\nu)=2^{-n}$ where $n=min\{k\geq0; \mu_{k}\neq \nu_{k}\}$. It is obvious that this classical metric has no equilateral triangles which means there exists no sequence $\eta(\eta\neq\mu\neq\nu)$ satisfying $d(\mu,\nu)=d(\mu,\eta)=d(\nu,\eta)$. Similarly, the new metric we introduced exists isosceles triangles. It indicates that there exist three expansive Lorenz maps $f,g,h$ admit $d(f,g)=d(f,h)\neq d(g,h)$.
\end{remark}
Next we will give an example to show the advantage of new metric over the classical metric.
\begin{example}\rm Suppose $f=(k_{+},k_{-})=((10001)^{\infty},(01100)^{\infty})$,
$g=(g_{+},g_{-})=(1000110001(110)^{\infty}$,$0110001100(01)^{\infty})$, $h=(h_{+},h_{-}) =(10001(100)^{\infty},01100(01)^{\infty})$. We can see that all  of the three kneading sequences are admissible, both $f$ and $h$ are renormalizable with $w_{+}=(100)$, $w_{-}=(01)$, only $g$ is prime. If endowed with the classical metric, we can obtain $d(f,g)=2^{-11}$ and $d(f,h)=2^{-7}$ which means $g$ is closer to $f$ than $h$ and this is in conflict with the fact. However, with respect to the new metric, $g$ is prime indicates $p=0$, $d(f,g)=\frac{1}{2^{0}}(1+\frac{1}{2^{11}}+\frac{1}{2^{12}})$. $f$ and $h$ have the same renormalization words means $p=1$ and $d(f,h)=\frac{1}{2^{1}}(1+\frac{1}{2^{3}}+\frac{1}{2^{3}})$. In this way, we can obtain $d(f,g)>d(f,h)$, $h$ is closer to $f$ than $g$, which satisfies the fact. Hence the new metric is more accurate.

\end{example}
\begin{remark}\rm Given $f,g,h\in L_{r}$, $f$ topologically conjugate to $g$ if and only if they admit the same sequences of pairs $(\beta_{i},\alpha_{i})$. Moreover, If $d(f,g)<d(f,h)$, then $g$ is more close to $f$ than $h$.
\end{remark}

\section{Uniform linearization}
\begin{definition}\rm
Suppose the kneading invariants of expansive Lorenz maps $K_{f}=({k_ + },{k_ - })$ can be renormalized via $w_{+}$ and $w_{-}$, and denote $RK_{f}$ be the new kneading invariants after renormalization. Set the polynomial $K(z,z) = {K_ + }(z,z) - {K_ - }(z,z)$ , where
$${K_ + }(z,z) = \sum\limits_{i = 0}^\infty  {{k_ + }(i){z^i}},\ \ {K_ - }(z,z) = \sum\limits_{i = 0}^\infty  {{k_ - }(i){z^i}}.$$
${k_ + }(i)$ means the $i$th symbol of sequence $k_{+}$. Set:
$${w_ + }(z,z) = \sum\limits_{i = 0}^{{n_ + }} {{w_ + }(i){z^i}}, \ \ \ {w_ - }(z,z) = \sum\limits_{i = 0}^{{n_ - }} {{w_ - }(i){z^i}}.$$
Where $|w_{+}|=n_{+}$ and $|w_{-}|=n_{-}$. Let $K({z^a},{z^b}) = {K_ + }({z^a},{z^b}) - {K_ - }({z^a},{z^b})$, where
$$k({z^a},{z^b}) = \sum\limits_{i = 0}^\infty  {{k_i}{z^{a(i - {l_i}) + b{l_i}}}}, \ \ \ {l_i} = \# \{ j|0 \le j < i,{k_j} = 1\}.$$

\end{definition}

\begin{proposition}
Let $f$ be an $m$-renormalizable expansive Lorenz map $ (0\leq m\leq\infty)$.
\item 1. $f$ is uniformly linearizable if and only if $m$ is finite and $\beta_{i}=1$ for $i=1,2\ldots m$. Moreover, $\beta^{\ast}=\beta_{m+1}^{1/(l_{1}l_{2}\ldots l_{m})}$, where $\beta^{\ast}$ corresponds to $(k_{+},k_{-})$, $l_{i}$ means the length of  $i$th periodic renormalization words.
\item 2. $h_{top}f=0$ if and only if $\beta_{i}=1$ for all $i$.
\end{proposition}

\begin{proof}\rm (1) A possible way to study the uniform piecewise linearizing questino is to the find suitable condition to exclude the existence of completely invariant Cantor set. For expanding Lorenz maps, a completely invariant Cantor set corresponds to a non-periodic renormalization. According to Proposition 1, if $\beta_{i}=1$ for $i=1,2\ldots m$, then all the $m$ times renormalization are perodic. In this way, $f$ admits no completely invariant Cantor set, so $f$ is conjugate to a linear mod one transformation. Now we will prove the expression of $\beta^{\ast}$ through inductive method.
\par When $m=1$, which means $f$ can be only periodic renormalized once. According to the definition of periodic renormalization$$ Rf=\left \{
\begin{array}{ll}
f^{l_{1}}(x) & x \in [a,\  c), \\
f^{l_{1}}(x) & x \in (c,\  b],
\end{array}
\right.
$$
Hence the entropy of $Rf$ satisfies $h(Rf)=h(f^{l_{1}})=l_{1}h(f)$, this indicates $\ln\beta_{1}=l_{1}\ln\beta^{*}$, then we can obtain $\beta^{*}=\beta_{1}^{1/l_{1}}$.
\par Now we suppose when $m=k$, the expression of $\beta^{*}=\beta_{k}^{1/(l_{1}l_{2}\ldots l_{k})}$ is accessible, then we will show $m=k+1$ also admits the expression. If $f$ can be $k+1$ times periodic renormalizable, then $h(R^{k+1}f)=h((R^{k}f)^{l_{k+1}})=l_{k+1}h(R^{k}f)$, which indicates $\ln\beta_{k+1}=l_{k+1}\ln\beta_{k}$, so we can obtain $\beta^{*}=\beta_{k+1}^{1/(l_{1}l_{2}\ldots l_{k+1})}$. The expression of $\beta^{*}$ is proved.

\noindent(2) If $\beta_{i}=1$ for all $i$, according to the expression in (1), we can obtain that $\beta^{*}=1^{1/(l_{1}l_{2}\ldots l_{k+1})}=1$, so $h_{top}f=\ln\beta^{*}=\ln1=0$. $\hfill\square$

\end{proof}

\begin{lemma}\rm
Suppose $(w_{+},w_{-})$ be the renormalization words of periodic renormalization, then only first two symbols of $w_{+}$ and $w_{-}$ are different, which means $w_{+}(z,z)-w_{-}(z,z)=1-z$.
\begin{proof}\rm We have known that each rational number corresponds a periodic renormalization $(w_{+},w_{-})$, $w_{+}$ and $w_{-}$ are on the same periodic orbit. According to Proposition 1, given a rational number $p/q$ ($p<q$, $p$ and $q$ are prime), there exist $(w_{+},w_{-})$ such that $p/q$ equals to the probability of 1 in the sequence. In addition, any periodic renormalization word can be obtained via trivial renormalization. For example, $(10,01)$ corresponds 1/2, after trivial renormalization $(10,0)$ or $(1,01)$ we can obtain 1/3 and 2/3. We call 1/3 and 2/3 be the preimages of 1/2, and each rational number have two preimages. Repeat the process, we can get a binary tree with initial number 1/2 and all the rational numbers between $(0,1)$ are in the binary tree. We first show any number in the binary tree, its two preimages' renormalization words has only two different symbols. Firstly, the renormalization words of preimages of 1/2 are $(100,010)$ and $(101,011)$, it is clear only first two symbols are not identical. Now given a rational number $p/q$, suppose its renormalization words $(10{a_1}{a_2}...{a_m},01{b_1}{b_2}...{b_m})$ has only two different symbols, we prove its preimages also satisfy the condition. If with trivial renormalization $(10,0)$, then we have $(100{({a_1}{a_2}...{a_m})^ * },010{({b_1}{b_2}...{b_m})^ * })$, satisfy the condition; if with trivial renormalization $(1,01)$, then we have $(101{({a_1}{a_2}...{a_m})^ * },011{({b_1}{b_2}...{b_m})^ * })$, also satisfy the condition. With the expression of $w(z,z)$, we obtain ${w_ + }(z,z) - {w_ - }(z,z) = 1 - t$ for each periodic renormalization words $(w_{+},w_{-})$. $\hfill\square$
\end{proof}
\end{lemma}
\begin{proposition} \rm
Suppose the kneading invariants $K_{f}=({k_ + },{k_ - })$ of $f \in {L_r}$ can be periodically renormalized via $w_{+}$ and $w_{-}$, and the length $|w_{+}|=|w_{-}|=n$, then

$${K_ + }(z,z) = \frac{{1 - t}}{2}(R{K_ + }({z^n},{z^n}) + R{K_ + }({z^n},{z^n})) + \frac{{{w_ - }(z,z)}}{{1 - {t^n}}}.$$
\begin{proof}\rm
Since $K(z,z) = {K_ + }(z,z) - {K_ - }(z,z) = 0$, there exists a real root $z\in(0,1)$ such that ${K_ + }(z,z) = {K_ - }(z,z)$, hence we can obtain ${K_ + }(z,z) = ({K_ + }(z,z) + {K_ - }(z,z))/2$. Suppose
$$\left\{ {\begin{array}{*{20}{c}}
{{k_ + } = w_ + ^{{a_0}}w_ - ^{{b_0}}w_ + ^{{a_1}}w_ - ^{{b_1}}....}\\
{{k_ - } = w_ - ^{{c_0}}w_ + ^{{d_0}}w_ - ^{{c_1}}w_ + ^{{d_1}}....}
\end{array}} \right.$$
where ${a_0} = {c_0} = 1$, next we have
$$\begin{array}{l}
{K_ + }(z,z) = {w_ + }(z,z) + \frac{{{z^{{a_0}n}}(1 - {z^{{b_0}n}})}}{{1 - {z^n}}}{w_ - }(z,z) + \frac{{{z^{{a_0}n + {b_0}n}}(1 - {z^{{a_1}n}})}}{{1 - {z^n}}}{w_ + }(z,z) + ....\\
 \ \ \ \  \ \ \  \ \ \ \ \ = \frac{{{w_ + }(z,z)}}{{1 - {z^n}}}[1 - {z^n} + \sum\limits_{j = 0}^\infty  {{z^{\sum\nolimits_{i = 0}^j {({a_i}n + {b_i}n)} }}(1 - {z^{{a_{j + 1}}n}})} ] + \\
 \ \ \ \  \ \ \  \ \ \ \ \ \frac{{{w_ - }(z,z)}}{{1 - {z^n}}}[{z^n} - \sum\limits_{j = 0}^\infty  {{z^{\sum\nolimits_{i = 0}^j {({a_i}n + {b_i}n)} }}(1 - {z^{{a_{j + 1}}n}})} ].
\end{array}$$
Similarly,
$$\begin{array}{l}
{K_ - }(z,z) = \frac{{{w_ - }(z,z)}}{{1 - {z^n}}}[1 - {z^n} + \sum\limits_{j = 0}^\infty  {{z^{\sum\nolimits_{i = 0}^j {({c_i}n + {d_i}n)} }}(1 - {z^{{c_{j + 1}}n}})} ] + \\
 \ \ \ \  \ \ \  \ \ \ \ \ \frac{{{w_ + }(z,z)}}{{1 - {z^n}}}[{z^n} - \sum\limits_{j = 0}^\infty  {{z^{\sum\nolimits_{i = 0}^j {({c_i}n + {d_i}n)} }}(1 - {z^{{c_{j + 1}}n}})} ].
\end{array}$$
Denote
$${Q_1} = \sum\limits_{j = 0}^\infty  {{z^{\sum\nolimits_{i = 0}^j {({a_i}n + {b_i}n)} }}(1 - {z^{{a_{j + 1}}n}})} , \ \ {Q_2} = \sum\limits_{j = 0}^\infty  {{z^{\sum\nolimits_{i = 0}^j {({c_i}n + {d_i}n)} }}(1 - {z^{{c_{j + 1}}n}})} .$$
Then obtain
$${K_ + }(z,z) = \frac{{{K_ + }(z,z) + {K_ - }(z,z)}}{2} = \frac{1}{2}(\frac{{{w_ + }(z,z)}}{{1 - {z^n}}} - \frac{{{w_ - }(z,z)}}{{1 - {z^n}}})(1 + {Q_1} - {Q_2}) + \frac{{{w_ - }(z,z)}}{{1 - {z^n}}}.$$
Calculate $RK({z^n},{z^n})$ with the same methods,
$$R{K_ + }({z^n},{z^n}) + R{K_ - }({z^n},{z^n}) = \frac{{{t^n}}}{{1 - {t^n}}}(1 + {Q_1} - {Q_2}).$$
Take into the expression of ${K_ + }(z,z)$:
$${K_ + }(z,z) = \frac{{1 - t}}{2}(R{K_ + }({z^n},{z^n}) + R{K_ + }({z^n},{z^n})) + \frac{{{w_ - }(z,z)}}{{1 - {t^n}}}.$$  $\hfill\square$

\end{proof}\end{proposition}
\begin{remark}\rm
The following two expressions of $\alpha$ are identical,
$$\left\{ {\begin{array}{*{20}{c}}
{\alpha  = (\beta  - 1)(\sum\limits_{i = 0}^\infty  {\frac{{{k_ + }(i)}}{{{\beta ^i}}} - 1),} }\\
{\alpha  = (\beta  - 1)\sum\limits_{i = 1}^\infty  {\frac{{{k_0}(i)}}{{{\beta ^i}}}.} }
\end{array}} \right.$$
In fact, $\sum\limits_{i = 0}^\infty  {\frac{{{k_ + }(i)}}{{{\beta ^i}}}}={K_ + }(z,z)$. Let $z=1/\beta$, we can obtain the new expression of $\alpha$:
$$\alpha  = (\beta  - 1)(\sum\limits_{i = 0}^\infty  {\frac{{{k_ + }(i)}}{{{\beta ^i}}} - 1)}  = (\frac{1}{z} - 1)[\frac{{1 - z}}{2}(R{K_ + }({z^n},{z^n}) + R{K_ + }({z^n},{z^n})) + \frac{{{w_ - }(z,z)}}{{1 - {z^n}}} - 1].$$
\end{remark}
Next we calculate two examples to verify the expression of $\alpha$.
\begin{example}\rm
Suppose $f \in {L_r}$ and the kneading invariants $K_{f}=({k_ + },{k_ - })=((100101)^{\infty},(0110)^{\infty})$ can be periodically renormalized via $w_{+}=(10)$ and $w_{-}=(01)$. Then $RK = ({(100)^\infty },{(01)^\infty })$, if calculate with previous expression of $\alpha$,
$$\alpha  = (\beta  - 1)\sum\limits_{i = 1}^\infty  {\frac{{{k_0}(i)}}{{{\beta ^i}}}}  = (\frac{1}{z} - 1)\frac{{{z^4} + {z^7} + {z^9}}}{{1 - {z^9}}} \approx 0.418876,$$
where $z$ is the smallest real root of $K(z,z)$. If calculate with the new expression,
$$\alpha  = (\frac{1}{z} - 1)[\frac{{1 - z}}{2}(\frac{1}{{1 - {z^6}}} + \frac{{{z^2}}}{{1 - {z^4}}}) + \frac{z}{{1 - {z^2}}} - 1] \approx 0.418876.$$
\end{example}
\begin{example}\rm
Suppose $f \in {L_r}$ and the kneading invariants $K_{f}=({k_ + },{k_ - })=((100010010)^{\infty},(010100)^{\infty})$ can be periodically renormalized via $w_{+}=(100)$ and $w_{-}=(010)$. Then $RK = ({(100)^\infty },{(01)^\infty })$, if calculate with previous expression,
$$\alpha  = (\beta  - 1)\sum\limits_{i = 1}^\infty  {\frac{{{k_0}(i)}}{{{\beta ^i}}}}  = (\frac{1}{z} - 1)\frac{{{z^4} + {z^7} + {z^9}}}{{1 - {z^9}}}\approx0.282122,$$
where $z$ is the smallest real root of $K(z,z)$. If calculate with the new expression,
$$\alpha  = (\frac{1}{z} - 1)[\frac{{1 - z}}{2}(\frac{1}{{1 - {z^9}}} + \frac{{{z^3}}}{{1 - {z^6}}}) + \frac{z}{{1 - {z^3}}} - 1]\approx0.282122.$$
\end{example}

\par Now we have obtained the new expression of $\alpha$ when the renormalization are periodic. Suppose $f \in {L_r}$ can be periodically renormalized $m$ times. Denote $(\beta,\alpha)$ be the corresponding parameter of initial kneading invariants, i.e., the kneading invariants before all the renormalizations. And $({\beta _{m + 1}},{\alpha _{m + 1}})$ be the corresponding parameter of $({R^m}{k_ + },{R^m}{k_ - })$, i.e., the kneading invariants after $m$ times renormalizations. The whole process follows:
$$({k_ + },{k_ - })\stackrel{w_{1}^{+},w_{1}^{-}}{\longrightarrow}(R{k_ + },R{k_ - }) \stackrel{w_{2}^{+},w_{2}^{-}}{\longrightarrow}\cdots \stackrel{w_{m}^{+},w_{m}^{-}}{\longrightarrow}({R^m}{k_ + },{R^m}{k_ - }).$$
Firstly, begin from $({\beta _{m + 1}},{\alpha _{m + 1}})$, use the new expression of $\alpha$ and $(w_{m}^{+},w_{m}^{-})$, we can obtain $\alpha_{m}$. Then regard as $({\beta _{m }},{\alpha _{m }})$ as the last level, and repeat the process, finally we can obtain the value of $\alpha$. Notice that here only consider the case that all renormalizations are periodic, since all the renormalization words are periodic and ${w_ + }(z,z) - {w_ - }(z,z) = 1 - z$ holds.

\

\section*{Acknowledgements}

We would like to thank the referees for reading this paper very carefully. They suggested many improvements to a previous version of this paper. Thanks Zhou and Wang for the help of matlab skills. This work is supported by the Excellent Dissertation Cultivation Funds of Wuhan University of Technology (2018-YS-077).
\\

\section*{References}
\setcounter{equation}{0}

%\newpage

\end{document}